\documentclass[11pt]{amsart}
\usepackage[latin2]{inputenc}
\usepackage{amsmath,amsthm}
\usepackage{amsfonts}
\usepackage{amssymb}
\usepackage{graphicx}
\usepackage{indentfirst}
\usepackage{hyperref}
\usepackage{mathtools}
\usepackage[margin=1in]{geometry}
\usepackage{alltt}
\usepackage{comment}
\usepackage{todonotes}
\usepackage{enumitem}

\newcommand{\Z}{\mathbb{Z}}
\newcommand{\Q}{\mathbb{Q}}

\numberwithin{equation}{section}

\newtheorem{thmalpha}{Theorem}

\newtheorem{thm}{Theorem}[section]
\newtheorem{cor}[thm]{Corollary}
\newtheorem{lem}[thm]{Lemma}

\newtheorem{que}[thm]{Question}
\newtheorem{prop}[thm]{Proposition}
\newtheorem{conj}[thm]{Conjecture}
\newtheorem*{thm*}{Theorem}
\newtheorem*{nota*}{Notation}

\theoremstyle{definition}
\newtheorem{remark}[thm]{Remark}
\newtheorem{defi}[thm]{Definition}

\def\pitem{\advance\leftskip3mm\advance\linewidth-3mm}
\def\mitem{\advance\leftskip-3mm\advance\linewidth3mm}

\begingroup\makeatletter\ifx\SetFigFont\undefined
\gdef\SetFigFont#1#2#3#4#5{
  \reset@font\fontsize{#1}{#2pt}
  \fontfamily{#3}\fontseries{#4}\fontshape{#5}
  \selectfont}
\fi\endgroup

\renewcommand{\subjclass}[1]{\thanks{\emph{2020 Mathematics Subject Classification:}~#1}}
\renewcommand{\keywords}[1]{\thanks{\emph{Keywords and Phrases:}~#1}}
\renewcommand{\date}{\thanks{\today}}

\newcommand{\MM}{\mathcal{M}}

\newcommand{\OO}{\mathcal{O}}

\newcommand{\fp}{\mathfrak{p}}

\newcommand{\Cc}{\mathbb{C}}

\newcommand{\Qq}{\mathbb{Q}}

\newcommand{\Zz}{\mathbb{Z}}

\newcommand{\Ff}{\mathbb{F}}

\newcommand{\nc}{\newcommand}
\nc{\on}{\operatorname}
\nc \defeq {\vcentcolon=}
\nc{\eqdef}{=\vcentcolon}
\nc{\PI}{\mathcal{PI}}
\nc{\MI}{\mathcal{MI}}
\nc{\renc}{\renewcommand}
\renc{\P}{\mathbb{P}}
\renc{\figurename}{Table}

\def\house#1{\setbox1=\hbox{$\,#1\,$}%
\dimen1=\ht1 \advance\dimen1 by 2pt \dimen2=\dp1 \advance\dimen2
by 2pt
\setbox1=\hbox{\vrule height\dimen1 depth\dimen2\box1\vrule}%
\setbox1=\vbox{\hrule\box1}%
\advance\dimen1 by .4pt \ht1=\dimen1 \advance\dimen2 by .4pt
\dp1=\dimen2 \box1\relax}

\newcommand{\kdots}{,\ldots ,}

\newcommand{\GL}{{\rm GL}}

\newcommand{\medfrac}[2]{\mbox{\large{$\textstyle{\frac{#1}{#2}}$}}}
\newcommand{\medbinom}[2]{\mbox{\large{$\textstyle{\binom{#1}{#2}}$}}}
\newcommand{\medmatrix}[1]{\mbox{$\Big($\footnotesize{$\begin{matrix}#1\end{matrix}$}$\Big)$}}

\renewcommand{\gcd}{{\rm gcd}}

\newcommand{\uX}{\underline{X}}

\makeatletter
\renewcommand\@biblabel[1]{#1.}
\makeatother

\author[M. Bhargava, J.-H. Evertse, K. Gy\H{o}ry, L. Remete, A. A. Swaminathan]{Manjul Bhargava, Jan-Hendrik Evertse, K\'alm\'an Gy\H{o}ry, \\ L\'aszl\'o Remete, and Ashvin A.~Swaminathan}

\address{Manjul Bhargava and Ashvin A.~Swaminathan\newline
Princeton University,
Department of Mathematics\newline
Fine Hall, Washington Road, Princeton NJ 08540, USA\newline
\textit{email addresses:}
{
\tt bhargava@math.princeton.edu} and {\tt ashvins@math.princeton.edu}}

\address{Jan-Hendrik Evertse\newline
Leiden University, Department of Mathematics\newline
P.O.Box 9512, 2300 RA Leiden, The Netherlands\newline
\textit{email address:} {\tt evertse@math.leidenuniv.nl}}

\address{K\'alm\'an Gy\H{o}ry and L\'aszl\'o Remete\newline
Institute of Mathematics, University of Debrecen\newline
P.O.Box 400, H-4002 Debrecen, Hungary\newline
\textit{email addresses:} {\tt gyory@science.unideb.hu} and {\tt remete.laszlo@science.unideb.hu}}

\title{\vspace*{-0.6in}Hermite equivalence of polynomials}
\thanks{The research of the first-named author was supported by a Simons Investigator Grant and NSF grant~DMS-1001828. The research of the third-named author was supported in part by Grants K115479 and K128088 from the Hungarian National Foundation for Scientific Research (OTKA) and from the Austrian-Hungarian joint project ANN130909 (FWF-NKFIH). 
The research of the fourth named author was supported in part by the project EFOP-3.6.1-16-2016-00022 co-financed by the European Union and the European Social Fund.
The research of the fifth-named author was supported by the NSF Graduate Research Fellowship.}
\dedicatory{To the memory of Professor Andrzej Schinzel (1937--2021)}

\begin{document}

\maketitle
\vspace*{-0.3in}

\section{Introduction}\label{section1}

\subsection{Summary}

In this paper, we resurrect a long-forgotten notion of equivalence for univariate polynomials with integral coefficients introduced by Hermite in the 1850s. We show that the Hermite equivalence class of a polynomial has a very natural interpretation in terms of the invariant ring and invariant ideal associated with the polynomial.
%Using this interpretation, we give a necessary and sufficient condition for two polynomials to be Hermite equivalent in terms of rings and ideals---one that is also checkable in practice.
We apply this interpretation to shed light on the relationship between Hermite equivalence and more familiar notions of polynomial equivalence, such as $\on{GL}_2(\Z)$- and $\Z$-equivalence. Specifically, we prove that $\GL_2(\Zz )$-equivalent polynomials are Hermite equivalent and, for polynomials of degree $2$ or $3$,
the converse is also true. On the other hand, for every $n\geq 4$, we give infinite collections of examples of polynomials $f,g\in \Zz [X]$ of degree~$n$ that are Hermite equivalent but not $\GL_2(\Zz )$-equivalent.

%we briefly outline the history of reduction theories of polynomials and binary forms
%over $\Zz$.
% polynomials and algebraic numbers with given discriminant over $\Z$.

Using his reduction theory for quadratic forms, Hermite proved (ineffectively) that polynomials in $\Zz [X]$ with given discriminant lie in finitely many Hermite equivalence classes (this was in fact the reason why Hermite introduced his notion of equivalence). In this paper, we also compare Hermite's finiteness theorem with the most important results of this area, due to Birch and Merriman~\cite{BM1972} (1972), Gy\H{o}ry~\cite{Gy1973,Gy0974}  (1973,1974) and Evertse and Gy\H{o}ry~\cite{EGy1991,EGy2017} (1991,2017), which imply in a precise and effective form that polynomials in $\Zz [X]$ of given discriminant lie in finitely many $\GL_2(\Zz )$-equivalence classes, and hence in finitely many Hermite equivalence classes.

We point out that these results of Birch and Merriman, Gy\H{o}ry, and Evertse and Gy\H{o}ry are much more precise than Hermite's theorem and require deeper tools to prove. In particular, we correct a faulty reference occurring in Narkiewicz' excellent book~\cite{N2019} (2019), where $\GL_2(\Zz )$-equivalence and Hermite equivalence of polynomials were mixed up.

\subsection{Background}

In the mid-nineteenth century, Hermite~\cite{H1854,H1857} introduced a new notion of equivalence---which we call \emph{Hermite equivalence}---for univariate polynomials with integral coefficients. His motivation was to prove a finiteness theorem for equivalence classes of polynomials having given degree and discriminant. Such finiteness theorems had already been proven for $\on{GL}_2(\Z)$-equivalence classes of quadratic polynomials by Lagrange~\cite{L1773}, whose work was later improved by Gauss~\cite{G1801}. Hermite~\cite{H1851} proved the same finiteness statement for $\on{GL}_2(\Z)$-equivalence classes of cubic polynomials. Although he was unable to extend this result to polynomials of degree greater than three, Hermite realized that, if one replaces ``$\on{GL}_2(\Z)$-equivalence'' with ``Hermite equivalence,'' the desired finiteness statement would follow from the reduction theory for quadratic forms that he had previously developed in~\cite{MR1578622}.

Using his reduction theory for quadratic forms, Hermite proved (ineffectively) that polynomials in $\Zz [X]$ with given discriminant lie in finitely many Hermite equivalence classes. Hermite's original objective---proving that there are finitely many $\on{GL}_2(\Z)$-classes of polynomials of given degree and discriminant---was finally achieved more than a century later by Birch and Merriman~\cite{BM1972},
and independently, for monic polynomials and in a more precise and effective form by Gy\H{o}ry~\cite{Gy1973}. The result of Birch and Merriman was subsequently made effective by Evertse and Gy\H{o}ry~\cite{EGy1991}. Surprisingly, Hermite's result on finiteness for Hermite equivalence classes was not mentioned in any of these works, or in the related papers of Delone~\cite{D1930}, Nagell~\cite{N1930}, Gy\H{o}ry~\cite{Gy0974,Gy1976,Gy1978a,Gy1978b,Gy1980,Gy1994,Gy1998}, and Evertse and Gy\H{o}ry~\cite{EGy1991}. In fact, Hermite equivalence of polynomials does not appear to have been studied in the literature in the nearly two centuries since Hermite first introduced~the~notion.

The purpose of this paper is twofold: (1) to provide a thorough treatment of the notion of Hermite equivalence, and (2) to compare Hermite equivalence with two more familiar notions of equivalence for univariate integral polynomials, namely, $\on{GL}_2(\Z)$-equivalence and $\Z$-equivalence. We present theoretical arguments as well as examples to shed light on the relationships among these three different types of polynomial equivalence.

%\textcolor[rgb]{0,0,1}{Specifically, we show that the Hermite equivalence class of a polynomial has a very natural interpretation in terms of the invariant ring and invariant ideal associated with the polynomial.
%Using this interpretation, we give a necessary and sufficient condition for two polynomials to be Hermite equivalent in terms of rings and ideals---one that is also checkable in practice.
%We apply this interpretation to prove that $\GL_2(\Zz )$-equivalent polynomials are Hermite equivalent and, for polynomials of degree $2$ or $3$,
%the converse is also true. On the other hand, for every $n\geq 4$, we give infinite collections of examples of polynomials $f,g\in \Zz [X]$ of degree~$n$ that are Hermite equivalent but not $\GL_2(\Zz )$-equivalent.}
%\textcolor[rgb]{0,0,1}{We also compare the results of Hermite and Birch--Merriman with the most important results of this area, due to  Gy\H{o}ry (1973,1974) and Evertse and Gy\H{o}ry (1991,2017), which imply in a precise and effective form that polynomials in $\Zz [X]$ of given discriminant lie in finitely many $\GL_2(\Zz )$-equivalence, and hence in finitely many Hermite equivalence classes.
%We point out that these results of Birch and Merriman, Gy\H{o}ry, and Evertse and Gy\H{o}ry are much more precise than Hermite's theorem and require deeper tools to prove. In particular, we correct a faulty reference occurring in Narkiewicz' excellent book (2019), where $\GL_2(\Zz )$-equivalence and Hermite equivalence of polynomials were mixed up.
%}

\subsection{Notions of equivalence} \label{sec-basicnotions}

We now define the three notions of equivalence for polynomials in $\Z[X]$ studied in this paper, namely, Hermite equivalence (\S\ref{sec-eqhermite}), $\on{GL}_2(\Z)$-equivalence (\S\ref{sec-eqgl2z}), and $\Z$-equivalence (\S\ref{sec-eqz}).

\subsubsection{Hermite equivalence} \label{sec-eqhermite}
 To a polynomial \begin{equation} \label{eq-form}
f(X)=f_0X^n+f_1X^{n-1}+\cdots +f_n=f_0(X-\alpha_1)\cdots (X-\alpha_n)\in \Zz [X],
\end{equation}
where $\alpha_1\kdots\alpha_n\in\Cc$, Hermite associated the following decomposable form $[f]$ in $n$ variables:
\begin{equation} \label{hdef}
[f](\underline{X})=
f_0^{n-1}\prod_{i=1}^n\left(\alpha_i^{n-1}X_1+\alpha_i^{n-2}X_2+\cdots+X_{n}\right),
\end{equation}
where $\underline{X}$ denotes the column vector $(X_1\kdots X_n)^T$. As we shall show, the form $[f]$ has integer coefficients; it is primitive (i.e., its coefficients have greatest common divisor $1$) if and only if $f$ is primitive; and its discriminant is equal to that of the polynomial $f$.

Using the above construction of the form $[f]$, Hermite introduced the following notion of equivalence, which we call \emph{Hermite equivalence}, for polynomials $f,g\in\Z[X]$ of degree $n$:
%Write
%$$f(X)=f_0(X-\alpha_1)\cdots(X-\alpha_n),\;\; g(X)=g_0(X-\beta_1)\cdots(X-\beta_n)$$
%and consider the decomposable forms
%$$[f](\underline{X}) \defeq f_0^{n-1}\prod_{i=1}^{n}\left(\alpha_i^{n-1}X_1+\alpha_i^{n-2}X_2+\cdots+X_{n}\right)$$
%for $\alpha_1, \dots, \alpha_n, \beta_1, \dots, \beta_n \in \Cc$, and
%$$[g](\underline{X}) \defeq g_0^{n-1}\prod_{i=1}^{n}\left(\beta_i^{n-1}X_1+\beta_i^{n-2}X_2+\cdots+X_{n}\right).$$
%Then $f$ and $g$ are said to be Hermite equivalent, if
\begin{defi}
Let $f,g\in\Z[X]$ be polynomials of degree $n$. Then $f$ and $g$ are said to be \textit{Hermite equivalent} if the decomposable forms $[f],[g]$ are $\GL_n(\Z)$-equivalent, i.e., if there is a matrix $U\in \GL_n(\Z)$ such that
$$[g](U\underline{X})=\pm[f](\underline{X}).$$
\end{defi}
\noindent Since the action of $\on{GL}_n(\Z)$ on homogeneous forms of degree $n$ in $n$ variables is discriminant-preserving, it follows that the discriminants of Hermite equivalent polynomials are equal.

%Let $F(X,Y)=Y^nf(\frac{X}{Y})$ and $G(X,Y)=Y^ng(\frac{X}{Y})$ be the binary forms in $\Z[X,Y]$ associated to $f,g$, respectively. The degree and discriminant of $F,f$ and $G,g$ coincide, the Hermite equivalence of $F$ and $G$ can be defined in the same way, and Hermite's result can be reformulated for binary forms.

\subsubsection{$\on{GL}_2(\Z)$-equivalence} \label{sec-eqgl2z}

As far as we know, Hermite did not compare his equivalence with the well-known notion of $\GL_2(\Z)$-equivalence. Recall that two binary $n$-ic forms (i.e., binary forms of degree $n$) $F,G\in\Zz [X,Y]$ are called $\GL_2(\Zz )$-\emph{equivalent} if
$G(X,Y)=\pm F(aX+bY,cX+dY)$ for some $\bigl(\begin{smallmatrix}
a&b\\c&d\end{smallmatrix}\bigr)\in\GL_2(\Z)$.
%Two binary $n$-ic forms are called {\it $\pm\GL_2(\Z)$-equivalent} if $F$ is $\GL_2(\Z)$-equivalent to $\pm G$.
%\[
%G(X,Y)=F(aX+bY,cX+dY)\ \ \text{for some }
%\left(\begin{array}{rr}
%a & b\\
%c & d
%\end{array}\right)\in \GL_2(\Zz ).
%\]
In this case $F$ and $G$ have the same discriminant.

%Every nonzero binary $n$-ic form with coefficients in $\Z$ is $\on{GL}_2(\Z)$-equivalent to a binary form of the shape $F(X,Y)=a_0X^n+a_1X^{n-1}Y+\cdots+a_nY^n$, where $a_0=F(1,0)\neq 0$.
%To each $F$ with $F(1,0)\neq 0$,
To a binary $n$-ic form $F$, we may associate the univariate polynomial $f(X)=F(X,1)$. The discriminants of $F$ and of $f$ (viewed as a degree $n$ polynomial) coincide.
Conversely, if $f\in\Z[X]$ is a polynomial of degree at most $n$, we can associate to $f$ (viewed as a polynomial of degree $n$) its homogenization, namely, the binary $n$-ic form $F(X,Y)=Y^nf\bigl(\frac{X}{Y}\bigr)$. We then define two polynomials $f,g\in \Zz [X]$ of degree $n$ to be $\GL_2(\Zz )$-equivalent if their homogenizations are $\GL_2(\Zz )$-equivalent. This means precisely that
%\[
%g(X)=\pm (cX+d)^nf\Big(\frac{aX+b}{cX+d}\Big)\ \ \text{for some }
%\left(\begin{array}{rr}
%a & b\\
%c & d
%\end{array}\right)\in \GL_2(\Zz ).
%\]
$g(X)=\pm (cX+d)^nf\bigl(\frac{aX+b}{cX+d}\bigr)$
for some $\bigl(\begin{smallmatrix}
a&b\\c&d\end{smallmatrix}\bigr)\in\GL_2(\Z)$.

We shall show in what follows that $\GL_2(\Z)$-equivalence implies Hermite equivalence, and further that Hermite equivalence is in general weaker than $\GL_2(\Z)$-equivalence. For the sake of convenience, we shall work mostly with $ \on{GL}_2(\Z)$-equivalence of univariate polynomials rather than of binary forms.

\subsubsection{$\Z$-equivalence} \label{sec-eqz}

Two monic polynomials $f,g\in\Z[X]$ of degree $n$ are said to be $\Z$\textit{-equivalent} if $g(X)=\varepsilon^nf(\varepsilon X+a)$ for some $\varepsilon\in\{ \pm 1\}$ and $a\in\Z$.
Clearly, $\Zz$-equivalent polynomials have the same discriminant, and $\Zz$-equivalence implies $\GL_2(\Zz )$-equivalence (and hence also Hermite equivalence, as we shall show). Note that $\Zz$-equivalence is in general much stronger than $\GL_2(\Z)$-equivalence.

\subsection{Main theorems}
%We now summarize the main results of this paper. To do this, we require some notation.

Given a polynomial $f \in \Z[X]$ of degree $n \geq 2$, define the \emph{invariant order} of $f$ to be the ring $R_f$ of global sections of the subscheme of $\mathbb{P}^1_{\Z}$ cut out by the homogenization of $f$ (i.e., the unique binary $n$-ic form $F$ such that $F(x,1) = f(x)$).  Define the \emph{invariant ideal} of $f$ to be the $R_f$-module $I_f$ of global sections of the pullback of the line bundle $\OO(1)$ from $\mathbb{P}^1_{\Z}$ to $\on{Spec} R_f$.

Explicitly, if  $f(X)=f_0X^n+f_1X^{n-1}+\cdots+f_n \in \Z[X]$ is a polynomial of degree $n$ with leading coefficient $f_0 \neq 0$, and
$\alpha$ is the residue class of $X$ in $K_f \defeq \Q[X]/(f)$, then $R_f\subset K_f$ is isomorphic to the ring with $\Z$-basis
$$1,\;\;f_0\alpha,\;\;  f_0\alpha^2+f_1\alpha,\;\; \ldots\;,\;\; f_0\alpha^{n-1}+f_1\alpha^{n-2}+\cdots+f_{n-2}\alpha$$
and $I_f$ is isomorphic to the fractional $R_f$-ideal generated by $1$ and $\alpha$. (See Birch--Merriman~\cite{BM1972}, Nakagawa~\cite{Nakagawa}, and Wood~\cite{Wood}.)
%For more concrete definitions of $R_f$ and $I_f$, refer to \S\ref{sec-rings}.)

\mbox{Then we have the following theoretical results:}

\begin{thm} \label{thm-main}
Let $n \geq 2$ be an integer. Then:
\begin{itemize}[leftmargin=30pt]
    \item[$(i)$]$(\text{\rm Corollary~\ref{cor-ringeq}})$ Two polynomials $f,g \in \Z[X]$ of nonzero discriminant and degree $n$ are Hermite equivalent if and only if their invariant orders $R_f$ and $R_g$ are isomorphic, and under such an isomorphism, the $(n-1)^{\text{st}}$ powers of their invariant ideals $I_f$ and $I_g$ belong to \mbox{the same ideal class.}
    \item[$(ii)$]$(\text{\rm Corollary~\ref{cor3.4}})$ Two monic polynomials $f,g \in \Z[X]$ of nonzero discriminant and degree $n$ are Hermite equivalent if and only if their invariant orders $R_f$ and $R_g$ are isomorphic.
    \item[$(iii)$]$(\text{\rm Corollary~\ref{Prop2.1}})$ If two polynomials $f,g \in \Z[X]$ of degree $n$ are $\GL_2(\Z)$-equivalent, then they are Hermite equivalent. In particular, if $f$ and $g$ are monic and $\Z$-equivalent, then they are Hermite equivalent.
\end{itemize}
\end{thm}
An important consequence of Theorem~\ref{thm-main}(iii) is that the effective finiteness theorems due to Evertse and Gy\H{o}ry~\cite{EGy1991,EGy2017} and Gy\H{o}ry~\cite{Gy1973,Gy0974} for $\on{GL}_2(\Z)$-equivalence or $\Z$-equivalence classes of polynomials of given discriminant (see Theorems~\ref{ThmC},~\ref{ThmCC},~\ref{ThmE}, and~\ref{ThmD} in \S\ref{sec-allthefiniteness} below) apply just as well to Hermite equivalence classes.  We also have effective bounds, due to Lagrange~\cite{L1773} (for $n = 2$), Levi--Delone--Faddeev~\cite{MR0160744} and Bennett~\cite{Ben2001} (for $n = 3$), Akhtari and Bhargava~\cite{quartics} (for $n = 4$), and Evertse and Gy\H{o}ry~\cite{EGy2017} and Evertse~\cite{E2011} (for general $n \geq 5$), for the number of ways in which a ring arises as the invariant order of a $\on{GL}_2(\Z)$- or $\Z$-equivalence class of polynomials. We thus obtain the following finiteness results. We use the notation $\log^* x \defeq \max (1,\log x)$ for $x>0$. 

\begin{thm} \label{thm-main2}
Let $n \geq 2$ be an integer. Then we have the following four points:
\begin{itemize}[leftmargin=30pt]
    \item[$(i)$]$(\text{\rm Theorems~\ref{ThmC} and~\ref{ThmD}})$ The number of Hermite equivalence classes of polynomials in $\Z [X]$ of given discriminant $D \neq 0$ is effectively bounded in a way that depends only on $D$. More specifically, every Hermite equivalence class of polynomials in $\Z[X]$ with degree $n$ and discriminant $D \neq 0$ \mbox{has a representative with coefficients not exceeding}
        $$\exp\left\lbrace\left(4^2n^3\right)^{25n^2}|D|^{5n-3}\right\rbrace$$
        in absolute value, and $n \leq 3 +  2 \log |D|/\log 3$.
    \item[$(ii)$]$(\text{\rm Theorems~\ref{ThmE} and~\ref{ThmD}})$  Every Hermite equivalence class of monic polynomials in $\Z[X]$ with degree $n$ and discriminant $D \neq 0$ \mbox{has a representative with coefficients not exceeding}
        $$\exp\left\lbrace n^{20}8^{n^2+19}\left(|D|(\log^*|D|)^n\right)^{n-1}\right\rbrace$$
        in absolute value, and $n \leq 2 + 2 \log |D|/\log 3$.
    \item[$(iii)$]$(\text{\rm Theorem~\ref{thm3.5}(i),(iii),(v),(vii)})$ The number of $\on{GL}_2(\Z)$-equivalence classes of separable polynomials in $\Z [X]$ of degree $n$ in a Hermite equivalence class is $1$ if $n = 2$ or $3$, at most $10$ if $n = 4$, and at most $2^{5n^2}$ if $n \geq 5$.
    \item[$(iv)$]$(\text{\rm Theorem~\ref{thm3.5}(ii),(iv),(vi),(viii)})$ The number of $\Z$-equivalence classes of monic separable polynomials in $\Z [X]$ of degree $n$ in a Hermite equivalence class is $1$ if $n = 2$, at most $10$ if $n = 3$, at~most~$2760$ if $n = 4$, and at most $2^{5n^2}$ if $n \geq 5$.
\end{itemize}
\end{thm}

%Version without explicit bounds
\begin{comment}
\begin{thm} \label{thm-main2}
Let $n \geq 2$ be an integer. Then we have the following four points:
\begin{itemize}[leftmargin=30pt]
    \item[$(i)$]$(\text{\rm Theorems~\ref{ThmC} and~\ref{ThmD}})$ The number of Hermite equivalence classes of polynomials of given discriminant $D \neq 0$ is effectively bounded in a way that depends only on $D$.
    %, namely, by
    %$$(n+1)^{\exp\left\lbrace\left(4^2n^3\right)^{25n^2}|D|^{5n-3}\right\rbrace}$$
    \item[$(ii)$]$(\text{\rm Theorems~\ref{ThmE} and~\ref{ThmD}})$ The number of Hermite equivalence classes of monic polynomials of given discriminant $D \neq 0$ is effectively bounded in a way that depends only on $D$.
    \item[$(iii)$]$(\text{\rm Theorem~\ref{thm3.5}(i),(iii),(v),(vii)})$ The number of $\on{GL}_2(\Z)$-equivalence classes of separable polynomials of degree $n$ in a Hermite equivalence class is effectively bounded in a way that depends only on $n$.
    \item[$(iv)$]$(\text{\rm Theorem~\ref{thm3.5}(ii),(iv),(vi),(viii)})$ The number of $\Z$-equivalence classes of monic separable polynomials of degree $n$ in a Hermite equivalence class is effectively bounded in a way that depends only on $n$.
\end{itemize}
\end{thm}
\noindent In Theorem 4.1, we will give explicit values for the bounds mentioned in Theorem~\ref{thm-main2}(iii),(iv).
\end{comment}

Finally, by constructing explicit examples, we prove the following existence theorems concerning the relationship between the aforementioned notions of polynomial equivalence:

\begin{thm} \label{thm-main3}
We have the following two points:
\begin{itemize}[leftmargin=30pt]
  \item[$(i)$]$(\text{\rm Theorem~\ref{thm-example}})$ There exist quartic polynomials $f,g \in \Z[X]$ of squarefree discriminant that have isomorphic invariant orders $R_f$ and $R_g$ but  $f$ and $g$ are not Hermite equivalent.
    \item[$(ii)$]$(\text{\rm \S\S\ref{sec-infmon4}--\ref{section4}})$ For each $n\geq 4$, there exist infinitely many Hermite equivalence classes of properly non-monic\footnote{I.e., not $\on{GL}_2(\Z)$-equivalent to a monic polynomial.} irreducible polynomials, of monic irreducible polynomials, and of monic reducible polynomials having degree $n$,  that split into more than one $\on{GL}_2(\Z)$-equivalence class.
    %\item[$(iii)$]$(\text{\rm Theorem~\ref{thm-redpols}})$ For each integer $n \geq 4$, there exist infinitely many distinct Hermite equivalence classes of monic reducible polynomials of degree $n$ that split into more than one $\on{GL}_2(\Z)$-equivalence class.
    %\item[$(iv)$]$(\text{\rm Theorem~\ref{thm-ftcandgtc}(i)})$ For each integer $n \geq 6$, there exist infinitely many distinct Hermite equivalence classes of properly non-monic\footnote{I.e., not $\on{GL}_2(\Z)$-equivalent to a monic polynomial.} irreducible polynomials of degree $n$ that split into more than one $\on{GL}_2(\Z)$-equivalence class.
    %\item[$(v)$]$(\text{\rm Theorem~\ref{thm-ftcandgtc}(ii)})$ For each integer $n \geq 7$, there exist infinitely many distinct Hermite equivalence classes of monic irreducible polynomials of degree $n$ that split into more than one $\on{GL}_2(\Z)$-equivalence class.
\end{itemize}
\end{thm}

Part (i) of Theorem~\ref{thm-main3} shows that the condition that the $(n-1)^{\text{st}}$ powers of $I_f$ and $I_g$ belong to the same ideal class cannot in general be dropped from Theorem~\ref{thm-main}(i).
Part (ii) shows that in all degrees $n\geq4$, the notion of Hermite equivalence is strictly weaker than $\GL_2(\Z)$-equivalence.

\begin{remark}
In the recent book of Narkiewicz~\cite{N2019}, on pp.~36--37, there is an incorrect, misleading reference, which suggests that Hermite proved that the polynomials with given discriminant lie in finitely many $\GL_2(\Zz )$-equivalence classes, when in fact Hermite had only proven this for Hermite equivalence classes. While Hermite could prove his finiteness result using his reduction theory of quadratic forms, the corresponding result for $\GL_2(\Z)$-equivalence classes requires much deeper tools not available to Hermite, namely, finiteness results for unit equations. One of our motivations in writing this article was to correct this misleading reference in Narkiewicz's book~\cite{N2019} and to illustrate by concrete examples that Hermite equivalence is in general weaker than the $\GL_2(\Z)$-equivalence and $\Z$-equivalence of polynomials.
\end{remark}

\subsection{Organization}
The rest of this paper is organized as follows. In \S\ref{sec-elem}, we prove a number of fundamental properties about Hermite equivalence of polynomials. In \S\ref{section3}, we give an interpretation of Hermite equivalence in terms of invariant orders and ideals, thus proving Theorem~\ref{thm-main}. We also use this intepretation to prove Theorem~\ref{thm-main3}(i). In \S\ref{sec-allthefiniteness}, we survey the literature on finiteness theorems for polynomial equivalence and observe that these theorems also apply to Hermite equivalence, thus proving Theorem~\ref{thm-main2}. We finish in \S\ref{sec-alltheexamples} by constructing the infinite collections of examples described in Theorem~\ref{thm-main3}(ii).

\section{Elementary considerations} \label{sec-elem}

In this section, we use elementary arguments to establish several important properties about Hermite equivalence of polynomials. In \S\ref{section3}, we demonstrate that these properties are straightforward consequences of our characterization of Hermite equivalence in terms of invariant orders and ideals.

\subsection{Content and primitivity}
Recall that the \emph{content} of a polynomial with integer coefficients is the positive greatest common divisor of its coefficients. A polynomial with integer coefficients is called primitive it its content is equal to $1$.

\begin{thm} \label{thm-prec}
Let $f\in \Z[X]$ be a polynomial of degree $n$ with content $c$. Then $[f]$ has integer coefficients, and its content is $c^{n-1}$.
\end{thm}
\begin{proof}
Let $K$ denote the splitting field of $f$. Denote by $(\alpha_1,\dots,\alpha_s)$ the fractional ideal of $\OO_K$ generated by $\alpha_1,\dots,\alpha_s\in K$. Given a polynomial $F$ with coefficients in $K$, denote by $(F)$ the fractional ideal with respect to $\OO_K$ generated by the coefficients of $F$. Then by Gauss' Lemma for Dedekind domains, we have $(FG)=(F) \cdot (G)$ for any two polynomials $F,G\in K[X_1,\dots,X_r]$.

Now write $f=f_0(X-\alpha_1)\cdots (X-\alpha_n)$, with $\alpha_1,\dots,\alpha_n\in K$ and $f_0\in \Z$. Then Gauss' Lemma implies that $(f)=(f_0)(1,\alpha_1)\cdots(1,\alpha_n)$ and $([f])=(f_0)^{n-1}(1,\alpha_1)^{n-1}\cdots(1,\alpha_n)^{n-1}=(f)^{n-1}$.
\end{proof}

\begin{remark}
Theorem~\ref{thm-prec} may also be proven quite explicitly, without relying on Gauss' Lemma. Indeed, if we write $\phi_{\underline{X}}(Y)=X_1Y^{n-1}+X_2Y^{n-2}+\cdots +X_n$, then one verifies that $[f]$ is simply the resultant of $\phi_{\uX}(Y)$ and $f(Y)$; i.e., we have that
\begin{equation*}%\label{eq2.2}
[f](\uX)={\rm Res}(\phi_{\uX},f)=
\left|\begin{array}{llllll}
X_1 &\ldots &X_n    &          &         &
\\
    &\ddots &       &\ddots    &         &
\\
    &       &\ddots &          &\ddots   &
\\
    &       &       &X_1       &\ldots   &X_n
\\
f_0 &\ldots &f_{n-1}&f_n       &         &
\\
    &\ddots &       &          &\ddots   &

\\
    &       &f_0    &\ldots    &f_{n-1}  &f_n
\end{array}\right| ,
\end{equation*}
where the first $n$ rows consist of $X_1\kdots X_n$ and the last $n-1$ rows
of $f_0\kdots f_{n-1},f_n$.
It follows that $[f]$ has integral coefficients. First assume that $f$ is primitive.
If $[f]$ is not primitive, then there is a prime $p$
such that $\on{Res}(\phi_{\underline{X}}, f) \equiv 0 \pmod p$. But then, $\phi_{\underline{X}}$ and $f$, viewed as polynomials over $\mathbb{F}_p[X_1, \dots, X_n]$, share a common factor. This happens if and only if $f \equiv 0 \pmod p$, which is impossible because $f$ is primitive. Hence $[f]$
is primitive as well.
Next, assume that $f$ has content $c$. Write $f'=f/c$. Then $f'$ is primitive and $[f]=c^{n-1}[f']$, which implies that
$[f]$ has content $c^{n-1}$.
%Note that the explicit argument provided above shows that $[f]$ has integer coefficients even when $D(f) = 0$. Alternatively, this can be deduced by combining the following two observations: (1) the coefficients of $[f]$ are elements of the ring $\Z[f_0, \dots, f_n][f_0^{-1}]$, and (2) by Corollary~\ref{cor-intnorm}, each of these coefficients takes on an integral value when $D(f) \neq 0$. These two observations imply that the coefficients of $[f]$ are in fact elements of the subring $\Z[f_0, \dots, f_n]$.
\end{remark}

\begin{cor}
Let $f,g\in \Z[X]$ be two non-zero polynomials with contents $c_f,\, c_g$ respectively,
and let $f'\defeq c_f^{-1}f$, $g'\defeq c_g^{-1}g$ be the corresponding primitive polynomials. 
Then $f$ and $g$ are Hermite equivalent if and only if $f'$ and $g'$ are Hermite equivalent and $c_g= c_f$.
\end{cor}
\begin{proof}
First assume that $f$ and $g$ are Hermite equivalent. So $[g](X)=\pm[f](UX)$ for some matrix $U\in \on{GL}_n(\Z)$. By Theorem~\ref{thm-prec}, $[f]$ has content $c_f^{n-1}$, and $[f](UX)$ has the same content as $[f]$. Further, $[g]$ has content $c_g^{n-1}$. So $c_g= c_f$.

Since $[f]=c_f^{n-1}[f']$ and $[g]=c_g^{n-1}[g']$ it follows that $[g'](X)=\pm[f'](UX)$.
Hence $f'$ and $g'$ are Hermite equivalent. The proof of the if-part is left to the reader.
\end{proof}

\subsection{Hermite equivalence and $\on{GL}_2(\Z)$-equivalence}
We now give an elementary and explicit proof that $\on{GL}_2(\Z)$-equivalence implies Hermite equivalence:

\begin{thm} \label{Prop2.1}
Let $f,g \in \Z[X]$ be $\on{GL}_2(\Z)$-equivalent polynomials. Then $f$ and $g$ are Hermite equivalent. In particular, if $f$ and $g$ are monic and $\Z$-equivalent, then they are Hermite equivalent.
\end{thm}
\begin{proof}
Write $f(X)=\prod_{i=1}^n (\alpha_{i,1}X-\alpha_{i,2})$. Then
\[
[f](\uX )=\prod_{i=1}^n (\alpha_{i,2}^{n-1}X_1+\alpha_{i,2}^{n-2}\alpha_{i,1}X_2+\cdots +\alpha_{i,1}^{n-1}X_n)=\prod_{i=1}^n \langle \underline{a}_i,\uX\rangle ,
\]
where $\underline{a}_i=(\alpha_{i,2}^{n-1}\kdots \alpha_{i,1}^{n-1})^T$
and $\langle\cdot ,\cdot\rangle$ denotes the standard inner product.
For $\gamma=\bigl(\begin{smallmatrix} a&b\\ c&d \end{smallmatrix}\bigr)$ write
\begin{align*}
\gamma f(X) &=(cX+d)^nf\bigl(\medfrac{aX+b}{cX+d}\bigr)
=\prod_{i=1}^n (\beta_{i,1} X-\beta_{i,2}),
\end{align*}
where $\beta_{i,1}=\alpha_{i,1}a-\alpha_{i,2}c$ and $\beta_{i,2}= -(-\alpha_{i,1}b +\alpha_{i,2}d)$. Then
\begin{align*}
[\gamma f](\uX ) &=\prod_{i=1}^n (\beta_{i,2}^{n-1}X_1+\beta_{i,2}^{n-2}\beta_{i,1}X_2+\cdots +\beta_{i,1}^{n-1}X_n)
\\
&=
\prod_{i=1}^n \langle t(\gamma )\underline{a}_i ,\uX\rangle
=\prod_{i=1}^n \langle \underline{a}_i, t(\gamma )^T\uX\rangle =[f](t(\gamma )^T\uX ),
\end{align*}
where $t(\gamma )$ is an $n\times n$ matrix whose entries are polynomials in $\Zz [a,b,c,d]$. One easily verifies that for any two $2\times 2$ matrices $\gamma_1,\gamma_2$ one has $t(\gamma_1\gamma_2)=t(\gamma_2)t(\gamma_1)$, and that $t(I_2)=I_n$, where $I_m$ is the $m\times m$ identity matrix. In particular, if $g=\pm \gamma f$ with $\gamma\in\GL_2(\Zz )$, then $[g](\uX )=\pm [f](t(\gamma )^T\uX )$, $t(\gamma )^T\in\GL_n(\Zz )$, so $f$ and $g$ are Hermite equivalent.
\end{proof}

\noindent As mentioned in Remark~\ref{rmk-2}, the converse of Theorem~\ref{Prop2.1} holds when $n = 2$. Combining the result of Levi--Delone--Faddeev stated in Remark~\ref{rmk-3} with Corollary~\ref{cor3.2} (to follow), we deduce that the converse also holds when $n = 3$. See \S\ref{section5} for examples of polynomials in every degree $n \geq 4$ for which the converse fails.

\subsection{Discriminant equalities}
The {\it discriminant} of a decomposable form of degree $n$ in $n$ variables
\[
F(\underline{X})=\prod_{i=1}^n (\gamma_{i,1}X_1+\cdots +\gamma_{i,n}X_n)
\]
is defined as
\[
D(F) \defeq \left( \det \big(\gamma_{i,j}\big)_{i,j=1\kdots n}\right)^2.
\]
Note that if $F,G$ are two decomposable forms of degree $n$ in $n$ variables
with $G(\uX)=\pm F(U\uX)$ for some $U\in\GL_n(\Zz )$, then $D(G)=D(F)$.

\begin{thm} \label{cor-discs}
Let $f \in \Z[X]$ be a polynomial of degree $n$. Then $D([f]) = D(f)$.
\end{thm}
\begin{proof}
The discriminant of $f$ is given by
\[
D(f) \defeq f_0^{2n-2}\prod_{1\leq i<j\leq n} (\alpha_i -\alpha_j)^2,
\]
and an easy application of Vandermonde's identity gives
the desired equality $D(f) =D([f])$.
\end{proof}

Since the action of $\on{GL}_n(\Z)$ on decomposable forms in $n$ variables is discriminant-preserving, we obtain the following immediate consequence of Theorem~\ref{cor-discs}:
\begin{cor} \label{cor-discs2}
Hermite equivalent polynomials have the same discriminant.
\end{cor}

\section{Interpretation of Hermite equivalence}\label{section3}

The purpose of this section is to give an interpretation of Hermite equivalence of integer polynomials $f$ and $g$ in terms of the invariant orders and ideals associated to $f$ and $g$, which we review in \S\ref{sec-rings}. We use this interpretation to prove Theorem~\ref{thm-main}.

\subsection{Rings and ideals associated to polynomials in $\Z[X]$} \label{sec-rings}

Let $f \in \Z[X]$ be a polynomial of degree $n$ as in~\eqref{eq-form} with $D(f) \neq 0$ and $f_0 \neq 0$. Consider the \'{e}tale $\Q$-algebra $K_f \defeq \Q[X]/(f(X))$, and let $\alpha$ be the image of $X$ in $K_f$. For $k \in \{0, \dots, n-1\}$, we define the free $\Z$-modules $I_f(k) \subset K_f$ with basis
\begin{equation}
\langle
1,\alpha,\ldots,\alpha^k,\zeta_{k+1},\ldots,\zeta_{n-1}
\rangle, \quad \text{where} \quad
\zeta_i=f_0\alpha^i+f_1\alpha^{i-1}+\cdots+f_{i-1}\alpha, \label{eq-zetadef}
\end{equation}
and we write $R_f \defeq I_f(0)$ and $I_f \defeq I_f(1)$.  Before we describe the basic properties of the $\Z$-modules $R_f$ and $I_f(k)$, we require the following definition of the norm of a fractional ideal of an order in an \'{e}tale algebra:

\begin{defi}
For a given order $\OO$ in an \'{e}tale $\Q$-algebra $K$ and fractional ideal $I \subset K$ of $\OO$, we define the norm $N_{\OO}(I)$ of $I$ with respect to $\OO$ to be the absolute value of the determinant of the $\Z$-linear transformation taking a $\Z$-basis of $\OO$ to a $\Z$-basis of $I$.
\end{defi}

\begin{remark}
 Note in particular that if $\alpha\in K$, then the norm of the fractional ideal $\alpha \OO$ is just the absolute value of the determinant of the $\Q$-linear map $x\mapsto \alpha x$, and thus, $N_{\OO}(\alpha \OO)= |N_{\Q}^K(\alpha )|$.
\end{remark}

\noindent The following theorem summarizes the basic properties of the $\Z$-modules $R_f$ and $I_f(k)$:

\begin{thm} \label{thm-everything}
With notation as above, the following properties hold:
\begin{itemize}[leftmargin=25pt]
\item[$(i)$] $R_f$ is a ring of rank $n$ over $\Z$ and thus an order in $K_f$;
\item[$(ii)$] $D(f) = D(R_f)$;
\item[$(iii)$] For each $k \in \{0, \dots, n-1\}$, the $\Z$-module $I_f(k)$ is an $R_f$-submodule of $K_f$ and hence a fractional ideal of $R_f$. We have that $I_f(k)$ is invertible if and only if $f$ is primitive;
\item[$(iv)$] $I_f(k) = I_f^k$, and $N_{R_f}(I_f(k)) = |f_0|^{-k}$;
\item[$(v)$] $I_f(n-2)$ is an explicit representative of the ideal class of
the ``inverse different'' or the ``dualizing module'' of $R_f$;
\item[$(vi)$] If $f' \in \Z[X]$ is primitive of degree $n$ such that $f' \mid f$, then the ring of $R_f$-module endomorphisms of $I_f(n-1)$ is isomorphic to $R_{f'}$ \mbox{$($i.e., $R_{f'} = \{\xi \in K_f : \xi I_f(n-1) \subseteq I_f(n-1)\}${}$)$; and}
\item[$(vii)$] $R_f$ is isomorphic to the ring of global functions on the subscheme of $\P^1_{\Z}$ cut out by the homogenization of $f$, and $I_f(k)$, as an $R_f$-module,
consists of the sections of the pullback of $\mathcal O(k)$ from $\P^1_{\Z}$.
Hence if $g$ is the translate of $f$ by some $\gamma \in \GL_2(\Z)$,
then the action of $\gamma$ on $\P^1_{\Z}$ induces an isomorphism between $R_f$ and $R_{g}$, and this isomorphism identifies the ideal classes $I_f$ and $I_{g}$. Explicitly, if $\gamma=\bigl(\begin{smallmatrix}
a & b\\ c& d\end{smallmatrix}\bigr)$, then
%\begin{equation}\label{ifif}
$I_{g}(k) = \left(
{-b\alpha+a}\right)^{-k} I_f(k).$
%\end{equation}
\end{itemize}
\end{thm}
\begin{proof}
Points (i) and (ii) are results of Birch and Merriman~\cite[proof of Lemma 3]{BM1972}) and Nakagawa~\cite[Proposition~1.1]{Nakagawa}. Points (iii) and (vii) are results of Wood~\cite[\S2.1 and Appendix A]{Wood} (see also Simon~\cite[\S3]{MR1968893}), and Point (iv) is an elementary calculation. Point (v) is a result of Simon~\cite[Proposition~14]{MR2523319}, and Point (vi) follows upon observing that $I_f(n-1) = I_{f'}(n-1)$ is an invertible fractional ideal for $R_{f'}$ by Point (iii).
\end{proof}
\noindent In what follows, we call the ring $R_f$ the {\it invariant order} of $f$, and we call the fractional ideal $I_f$ the {\it invariant $($fractional$)$ ideal} of $f$.

\begin{remark}
Although this paper is largely concerned with polynomials of nonzero discriminant, the construction of $R_f$ and its associated fractional ideals $I_f(k)$ can also be carried out canonically for polynomials having discriminant zero, and even for the zero polynomial; see
Wood~\cite[\S\S2.3--2.4]{Wood}.
\end{remark}

\begin{remark} \label{rmk-princip}
Note that when $f_0 = 1$, the ring $R_f$ is simply the monogenic order $\Z[\alpha]$ generated by $\alpha$, and the ideals $I_f(k)$ are all equal to the unit ideal (in particular, they are all principal).
\end{remark}
%Note that the ring $R_f$ and the ideals $I_f(k)$ come with natural
%bases.  We thus call the $I_f(k)$ {\it based ideals} of $R_f$, i.e.,
%ideals of $R_f$ of rank $n$ over $\Z$ equipped with an ordered
%$\Z$-basis.

\subsection{Interpretation of $[f]$ as the norm form of $I_f(n-1)$}

Let $K$ be an \'etale $\Q$-algebra of degree $n$. Then we can write $K=\prod_{j=1}^m K_j$, where
$K_1,\ldots ,K_m$ are number fields. Letting $\pi_j:\, K\to K_j$ be the projection
onto the $j$-th factor,
and $\sigma_{j,k}$ ($k=1,\dots , n_j\defeq [K_j:\Q ]$) the embeddings of $K_j$ in $\overline{\Q}$,
the trace and norm over $\Q$ of $\alpha\in K$ are equal to
\begin{align*}
Tr_{\Q}^K(\alpha )&=\sum_{j=1}^mTr_{\Q}^{K_j}(\pi_j(\alpha ))=\sum_{j=1}^m\sum_{k=1}^{n_j} \sigma_{j,k}\pi_j (\alpha ),
\\
N_{\Q}^K(\alpha )&=\prod_{j=1}^mN_{\Q}^{K_j}(\pi_j(\alpha ) )=\prod_{j=1}^m\prod_{k=1}^{n_j}\sigma_{j,k}\pi_j (\alpha ).
\end{align*}
The norm is naturally extended to polynomials in $K[X_1,\ldots ,X_n]$.
Thus, the norm over $\Q$ of such a
polynomial has its coefficients in $\Q$.

Let $\OO$ be an order in $K$, and let $I\subset K$ be a (not-necessarily-invertible) fractional ideal of $\OO$ having $\Z$-rank $n$. Then the {\it norm form} of $I$ with respect to $\OO$ and to the $\Z$-basis $\langle\alpha_1,\ldots,\alpha_n\rangle$ of $I$ is the decomposable integral form of degree $n$ in $n$ variables defined by
$$N_{I,\OO}(X_1,\ldots,X_n)\defeq\frac{N^K_\Q(\alpha_1 X_1+\cdots+\alpha_n X_n)}{N_{\OO}(I)}.
$$
This depends on the choice of a $\Z$-basis for $I$, but the $\on{GL}_n(\Z)$-equivalence class of $N_{I,\OO}$ is clearly independent of the choice of a basis.

The following lemma determines the discriminant of the norm form of a fractional ideal:
\begin{lem} \label{lem-discs}
With notation as above,
%suppose that $I$ is invertible. Then
we have that $D(N_{I,\OO}) = D(\mathcal{O})$.
\end{lem}
\begin{proof}
Choose a $\Z$-basis $\langle\omega_1,\ldots ,\omega_n\rangle$ of $\OO$, and let $\gamma \in \on{GL}_n(\Q)$ be the $\Q$-linear transformation taking $\omega_1,\ldots ,\omega_n$ to $\alpha_1,\ldots ,\alpha_n$.
From the definition of discriminant of a decomposable form and the expression for the trace mentioned above, it follows that the discriminant of $N_{\Q}^K(\alpha_1X_1+\cdots +\alpha_nX_n)$ is precisely
the discriminant $D(\alpha_1,\ldots ,\alpha_n)$ of the basis $\langle\alpha_1,\ldots ,\alpha_n\rangle$.
Thus,
\begin{equation*}
D(N_{I,\OO}) =N_{\OO}(I)^{-2}D(\alpha_1,\ldots ,\alpha_n)
=|\det\gamma |^{-2}D(\alpha_1,\ldots ,\alpha_n)=D(\omega_1,\ldots ,\omega_n)=D(\OO ). \qedhere
\end{equation*}
\end{proof}

%, where we define $N_J$ with respect to the basis $\kappa^{-1}\alpha_1, \dots, \kappa^{-1}\alpha_n$ of $J$.

The significance of the norm form of an ideal class is contained in the following theorem.
\begin{thm}\label{gennormformthm}
For each $i \in \{1, 2\}$, let $I_i$ be a fractional ideal of $\Z$-rank $n$ for an order $\mathcal{O}_i$ in an \'etale $\Q$-algebra~$K_i$ of degree $n$. If the norm forms $N_{I_1,\OO_1}(X_1,\ldots,X_n)$ and $N_{I_2,\OO_2}(X_1,\ldots,X_n)$
are $\GL_n(\Z)$-equivalent, then $K_1$ and $K_2$ are isomorphic, and under such an isomorphism, $I_1$ is identified with $\kappa I_2$ for some $\kappa \in K_2^\times$. Conversely, if there is an isomorphism $\varphi \colon K_1 \to K_2$ that identifies $\OO_1$ with $\OO_2$ and $I_1$ with $\kappa I_2$ for some $\kappa \in K_2^\times$, then $N_{I_1,\OO_1}(X_1, \dots, X_n)$ and $N_{I_2,\OO_2}(X_1, \dots, X_n)$ are $\on{GL}_n(\Z)$-equivalent.
\end{thm}
\begin{proof}
%First, suppose that we have an isomorphism $\sigma \colon K_1 \overset{\sim} \longrightarrow K_2$ under which $I_1$ is identified with $\kappa I_2$ for some $\kappa \in K_2^\times$. Then we may choose a $\Z$-basis $\alpha_1^{(i)},\ldots, \alpha_n^{(i)}$ of $I_i$ for each $i$ such that $\alpha_j^{(1)}=\kappa\sigma(\alpha_j^{(2)})$ for all $j$. %and let $\alpha_i = \kappa\sigma(\beta_i)$ for each $i$.
%It follows that
%\begin{align*}
%N_{I_2}(X_1,\ldots,X_n) & = \frac{N^{K_2}_\Q\big(\kappa  \sigma(\alpha_1^{(2)} X_1+\cdots+ \alpha_n^{(2)} X_n)\big)}{N(\kappa  \sigma(I_2))} \\
%& = \frac{N^{K_1}_\Q\big(\alpha_1^{(1)} X_1+\cdots+ \alpha_n^{(1)} X_n\big)}{N(I_1)} = N_{I_1}(X_1,\ldots,X_n).
%\end{align*}
%and since the list $\alpha_1, \dots, \alpha_n$ is a $\Z$-basis of $I$, the right-hand side of~\eqref{eq-normtransform} is $\pm \on{GL}_n(\Z)$-equivalent to $N_I(x_1, \dots, x_n)$.
%For the other direction, c
We prove only the first part of the statement; the second is left to the reader. Choose a $\Z$-basis $\langle\alpha_1^{(i)}, \dots, \alpha_n^{(i)}\rangle$ of $I_i$ for each $i$ so that, with respect to these bases, we have
$$N_{I_1,\OO_1}(X_1, \dots, X_n) = \pm N_{I_2,\OO_2}(X_1, \dots, X_n).$$
For each $i \in \{1,2\}$, write $K^{(i)} = \prod_{j = 1}^{m_i} K_j^{(i)}$, where $K_j^{(i)}$ is a number field for each $j$, and let $\pi_j^{(i)} \colon K^{(i)} \to K_j^{(i)}$ denote the projection map onto the $j^{\mathrm{th}}$ factor. Then, using the symbol ``$\propto$'' to denote ``equal up to multiplication by an element of $\Q^\times$,'' we find that
%\small
\begin{equation} \label{eq-twoways}
\prod_{j = 1}^{m_1} N_{\Q}^{K_j^{(1)}}\big(\pi_j^{(1)}(\alpha_1^{(1)})X_1+\cdots+\pi_j^{(1)}(\alpha_n^{(1)})X_n\big) \propto \prod_{j = 1}^{m_2} N_{\Q}^{K_j^{(2)}}\big(\pi_j^{(2)}(\alpha_1^{(2)})X_1+\cdots+\pi_j^{(2)}(\alpha_n^{(2)})X_n\big).
\end{equation}
\normalsize
Since $\alpha_1^{(i)}, \dots, \alpha_n^{(i)}$ is a $\Q$-basis of $K_i$, we see that $\pi_j^{(i)}(\alpha_1^{(i)}), \dots, \pi_j^{(i)}(\alpha_n^{(i)})$ is a $\Q$-spanning set of $K_j^{(i)}$ for each $j \in \{1, \dots, m_i\}$. Consequently, we have an equality $\{K_j^{(1)} : j \in \{1, \dots, m_1\}\} = \{K_j^{(2)} : j \in \{1, \dots, m_2\}\}$ of multisets, and so we can take $K^{(1)} = K^{(2)} \eqdef K$, $m_1 = m_2 \eqdef m$, $K_j^{(1)} = K_j^{(2)} \eqdef K_j$, and $\pi_j^{(1)} = \pi_j^{(2)} \eqdef \pi_j$.

Now, by permuting isomorphic factors among the fields $K_1,\dots, K_m$ if necessary, we have for each $j$ that
\begin{equation} \label{eq-isolateKi}
N_{\Q}^{K_j}\big(\pi_j(\alpha_1^{(1)})X_1+\cdots+\pi_j(\alpha_n^{(1)})X_n\big) \propto N_{\Q}^{K_j}\big(\pi_j(\alpha_1^{(2)})X_1+\cdots+\pi_j(\alpha_n^{(2)})X_n\big).
\end{equation}
%Now, the ring $\overline{\Q}[x_1, \dots, x_n]$ has unique factorization up to scalar multiples of the factors, so
Note that the constant of proportionality in~\eqref{eq-isolateKi} must be a norm from $K_j$, as can be seen by specializing $X_1, \dots, X_n$ to values in $\Q$. Then, since $\pi_j(\alpha_1^{(1)})X_1 + \cdots + \pi_j(\alpha_n^{(1)})X_n$ and $\pi_j(\alpha_1^{(2)})X_1 + \cdots + \pi_j(\alpha_n^{(2)})X_n$ respectively divide the left- and right-hand sides of~\eqref{eq-isolateKi} (as polynomials over $K_j$),
there must be some $\kappa_j \in K_j^\times$
and some automorphism $\sigma_j$ of $K_j$ such that $\pi_j(\alpha_k^{(1)}) = \kappa_j \sigma_j(\pi_j(\alpha_k^{(2)}))$ for each $j=1,\ldots, m$, $k=1,\ldots ,n$. Taking $\kappa = \prod_{j = 1}^m \kappa_j$ and $\sigma = \prod_{j = 1}^m \sigma_j$, we deduce that $\alpha_k^{(1)} = \kappa \sigma(\alpha_k^{(2)})$ for each $k$. Thus, $I_1 = \kappa \sigma(I_2)$, as desired.
\end{proof}

%\begin{proof}
%Note that any  decomposable $n$-ary $n$-ic form $F(x_1,\ldots,x_n)$ (that arises as the norm form of an ideal class of an order $\OO$ of $K$) is $\GL_n(\C)$-equivalent to $x_1x_2\cdots x_n$.
%
%Suppose that $I$ and $J$ are ideals of $\OO$ that have $\pm\GL_n(\Z)$-equivalent norm forms.  By changing the choice of $\Z$-basis for $I$, we may assume that the norm forms of $I$ and $J$ are the same up to sign, i.e.,  $N_I(x_1,\ldots,x_n)=\pm N_J(x_1,\ldots,x_n)$.
%
%Now, by applying a transformation $U\in\GL_n(\C)$ to both $N_I$ and $N_J$,
%\end{proof}

\noindent
Our next theorem states that when $f \in \Z[X]$ is a polynomial of degree $n$, the form $[f]$, as defined in~\eqref{hdef}, may be interpreted as a norm form of a fractional ideal, namely, $I_f(n-1)$.
\begin{thm}\label{normformthm}
Let $f \in \Z[X]$ be a polynomial of degree $n$. Then $[f]$ is, up to sign, the norm form of $I_f(n-1)$ with respect to $R_f$ and the power basis $(\ref{eq-zetadef})$.
\end{thm}
\begin{proof}
This follows from the definition of $[f]$ upon noting that $I_f(n-1)$ has $\Z$-basis $\langle 1, \alpha, \dots, \alpha^{n-1}\rangle$ and that the norm of $I_f(n-1)$ with respect to $R_f$ is equal to $|f_0|^{1-n}$.
\end{proof}

\begin{remark}
Theorem~\ref{cor-discs} may also be proven using Theorem~\ref{normformthm}. When $D(f) \neq 0$ and $f$ is primitive, combining Theorem~\ref{normformthm} with Lemma~\ref{lem-discs} and Theorem~\ref{thm-everything}(iii)--(iv) yields that $D([f]) = D(N_{I_f(n-1),R_f}) = D(R_f) = D(f)$. To include the case $D(f)=0$ and/or $f$ imprimitive, we observe that $D([f]) - D(f)$, viewed as a polynomial in the coefficients of $f$,
must be identically zero since it vanishes already if $D(f) \neq 0$ and $f$ is primitive.
\end{remark}

\subsection{Interpretation of Hermite equivalence in terms of invariant orders and ideals}

We may now prove the following necessary and sufficient criterion for Hermite equivalence:
\begin{thm} \label{thm-equivalence}
Let $f,g \in \Z[X]$ be polynomials of degree $n$ and nonzero discriminant.
%the integral $n$-ary $n$-ic forms $[f]$ and $[g]$
If $f$ and $g$ are Hermite equivalent,
then there is a $\Q$-algebra isomorphism from $K_f$ to $K_g$ that maps  $I_f(n-1)$ to $\kappa I_g(n-1)$ for some $\kappa \in K_g^\times$, and any such isomorphism maps $R_f$ to $R_g$.

Conversely, if there is a $\Q$-algebra isomorphism from $K_f$ to $K_g$ that maps $R_f$ to $R_g$
and $I_f(n-1)$ to $\kappa I_g(n-1)$ for some $\kappa \in K_g^\times$, then $f$ and $g$ are Hermite equivalent.
\end{thm}
\begin{proof}
Assume that $f,g$ are Hermite equivalent and write $f=c_ff'$, $g=c_gg'$,
where $f',g'\in\Z [X]$ are primitive polynomials and $c_f,c_g$ positive integers.
Theorems~\ref{gennormformthm} and \ref{normformthm} imply that there is a
$\Q$-algebra isomorphism from $K_f$ to $K_g$ that maps $I_f(n-1)$ to $\kappa I_g(n-1)$ for some $\kappa \in K_g^\times$.
By Theorem~\ref{thm-everything}(vi), the endomorphism rings of $I_f(n-1)$ and $I_g(n-1)$ are respectively isomorphic to $R_{f'}$ and $R_{g'}$. It follows that the isomorphism $K_f \overset{\sim}\longrightarrow K_g$ restricts to an isomorphism $R_{f'} \overset{\sim}\longrightarrow R_{g'}$.

Now Theorem~\ref{thm-prec} implies that
%by inspecting~\eqref{eq-zetadef}, one checks that $R_f = \Z + c_f R_{f'}$ and $R_g = \Z + c_g R_{g'}$, so since $D(R_f) = D(R_g)$ by Theorem~\ref{cor-discs}, and since $R_{f'}$ and $R_{g'}$ are isomorphic,
%we deduce that
$c_f = c_g$.
%Since $I_f(n-1)$ and $I_g(n-1)$ are invertible by part (iii) of Theorem~\ref{thm-everything}, their endomorphism rings are respectively isomorphic to $R_f$ and $R_g$.
Thus, since $R_f = \Z + c_f R_{f'}$ and $R_g = \Z + c_g R_{g'}$, the isomorphism $R_{f'} \overset{\sim}\longrightarrow R_{g'}$ restricts to an isomorphism $R_f \overset{\sim}\longrightarrow R_g$.

The second statement follows directly from Theorems~\ref{gennormformthm} and \ref{normformthm}.
\end{proof}

%, which establishes the second statement. The reverse direction of the first statement then follows %because $[f]$ and $[g]$ are, up to the action of $\on{GL}_n(\Z)$ and sign, both equal to the norm form %of the same ideal class of the same order.
%Applying part~(v) of Theorem~\ref{thm-everything}, we see that the classes of $I_f^{n-2}$ and $I_f^{n-1}$ are carried under this isomorphism to those of $I_g^{n-2}$ and $I_g^{n-1}$, respectively; hence this isomorphism carries the class of $I_f$ to that of $I_g$.
%\end{proof}
\noindent We have the following pithy rephrasing of Theorem~\ref{thm-equivalence} in terms of ideal classes of invariant orders:

\begin{cor} \label{cor-ringeq}
Let $f,g \in \Z[X]$ be polynomials of degree $n$ and nonzero discriminant. Then $f$ and $g$ are Hermite equivalent if and only if their invariant orders $R_f$ and $R_g$ are isomorphic, and under such an isomorphism, $I_f(n-1)$ and $I_g(n-1)$ belong to the same ideal class.
\end{cor}
\noindent Retain the setting of Corollary~\ref{cor-ringeq}, and suppose further that $f$ and $g$ are primitive. Applying part~(v) of Theorem~\ref{thm-everything}, we see that the ideal classes of $I_f^{n-2}$ and $I_f^{n-1}$ are carried under this isomorphism to those of $I_g^{n-2}$ and $I_g^{n-1}$, respectively; hence this isomorphism carries the class of $I_f$ to that of $I_g$. Thus, for primitive polynomials, we obtain the following variant of Corollary~\ref{cor-ringeq}:
\begin{cor} \label{cor-ringeqprim}
Let $f,g \in \Z[X]$ be primitive polynomials of nonzero discriminant. Then $f$ and $g$ are Hermite equivalent if and only if their invariant orders $R_f$ and $R_g$ are isomorphic, and under such an isomorphism, $I_f$ and $I_g$ belong to the same ideal class.
\end{cor}

\begin{remark} \label{rmk-2}
When $n = 2$, Theorem~\ref{thm-equivalence} is well-known. Indeed, if $f$ is a binary quadratic form, then $[f](\underline{X}) = f(X_2,-X_1)$, and so Hermite equivalence and $\on{GL}_2(\Z)$-equivalence are the same notion; since it is known (by the ideal class interpretation of Gauss composition) that $\GL_2(\Z)$-classes of integral binary quadratic forms $f$ are in bijection with isomorphism classes of pairs $(R,I)$, where $R$ is the invariant order of $f$ and $I$ is the invariant ideal of $f$, the result follows.
\end{remark}

\begin{remark} \label{rmk-3}
When $n = 3$, the condition that the invariant ideals $I_f$ and $I_g$ lie in the same ideal class can be dropped: by the Levi--Delone--Faddeev correspondence~\cite{MR0160744}, binary cubic forms define isomorphic rings if and only if they are $\on{GL}_2(\Z)$-equivalent.  In \S\ref{sec-needed}, we will show that, in contrast to the case $n = 3$, the condition in Theorem~\ref{thm-equivalence} that the invariant ideals $I_f$ and $I_g$ lie in the same ideal class {cannot} be dropped when $n>3$.
\end{remark}

\subsection{Consequences for Hermite equivalence}

In this subsection, we present several corollaries of Theorem~\ref{thm-equivalence} concerning Hermite equivalence.
%\noindent
First, by dropping the condition on the invariant ideals in Theorem~\ref{thm-equivalence}, we obtain the following consequence:
\begin{cor}\label{cor3.2}
Let $f,g \in \Z[X]$ be Hermite equivalent polynomials of nonzero discriminant.
Then their invariant orders $R_f$ and $R_g$ are isomorphic.
\end{cor}

\noindent The converse of Corollary~\ref{cor3.2} holds in certain special situations. For example, when $f$ and $g$ are both monic, their invariant ideals $I_f$ and $I_g$ are both principal (in fact, as stated in Remark~\ref{rmk-princip}, they are both equal to the unit ideal), and applying Theorem~\ref{thm-equivalence} yields the following consequence:

\begin{cor}\label{cor3.4}
Let $f,g\in \Z[X]$ be monic polynomials of nonzero discriminant. Then $f$ and $g$ are Hermite equivalent if and only if their invariant orders $R_f$ and $R_g$ are isomorphic.
\end{cor}

\begin{remark} \label{rmk-7}
Recall that the map sending the $\Z$-equivalence class of a monic quadratic polynomial $f \in \Z[X]$ to the unique quadratic ring with discriminant $D(f)$ is a bijection. Combining this fact with Corollary~\ref{cor3.4}, we see that two monic quadratic polynomials $f,g \in \Z[x]$ are Hermite equivalent if and only if they are $\Z$-equivalent.
\end{remark}

\noindent From the discussion in Remarks~\ref{rmk-2}--\ref{rmk-3}, we see that the converse to Corollary~\ref{cor3.2} does \emph{not necessarily} hold when $n = 2$, but that it \emph{does} hold when $n = 3$. In light of this, we pose the following question concerning the converse of Corollary~\ref{cor3.2}:
\begin{que}\label{que3.3}
Let $n \geq 4$. Do there exist $($reducible or irreducible$)$ polynomials $f,g \in \Z[X]$ of degree $n$ and nonzero discriminant that have isomorphic invariant orders but are not Hermite equivalent, and if so, can such polynomials be exhibited?
\end{que}
\noindent Corollary~\ref{cor3.4} implies that, in Question~\ref{que3.3}, if such polynomials $f,g$ exist they necessarily have to be non-monic. See \S\ref{sec-needed} for an explicit example in the quartic case.

\begin{remark} \label{rmk-disc0}
Theorem~\ref{Prop2.1} can also be proven using Theorem~\ref{thm-equivalence}. Indeed, observe that by Theorem~\ref{thm-everything}(vii), if $f$ and $g$ are $\on{GL}_2(\Z)$-equivalent, then the invariant orders $R_f$ and $R_g$ are naturally isomorphic and the invariant ideals $I_f$ and $I_g$ lie in the same ideal class under this isomorphism. By Theorem~\ref{thm-equivalence}, we conclude that if $f$ and $g$ have nonzero discriminant and are $\on{GL}_2(\Z)$-equivalent, then they are Hermite equivalent. In particular, if $f$ and $g$ are monic and $\Z$-equivalent, then they are Hermite equivalent.
\end{remark}

\subsection{Necessity of the condition on invariant ideals  when $n=4$} \label{sec-needed}

In this section, we give an answer to Question~\ref{que3.3} by showing that, in contrast to the case $n = 3$, the condition in Theorem~\ref{thm-equivalence} that the invariant ideals $I_f$ and $I_g$ lie in the same ideal class {cannot necessarily} be dropped when $n>3$.
Specifically, we consider the case $n=4$, and show the existence of two quartic polynomials $f,g \in \Z[X]$ such that $R_f$ is isomorphic to $R_g$ but $I_f$ and $I_g$ do not lie in the same ideal class under any isomorphism between $R_f$ and $R_g$.  We prove the following theorem, which implies Theorem~\ref{thm-main2}(i):

%Theorem~\ref{thm-equivalence} naturally leads to the following question, which was posed by Evertse, Gy\"{o}ry, and Remete:
%\begin{question}[\protect{\cite[Open problem~3.5]{EGR}}] \label{quest-open}
%Is it possible to construct binary $n$-ic forms $f$ and $g$ such that $R_f$ and $R_g$ are isomorphic but are not Hermite equivalent?
%\end{question}\todo{\small Isn't the answer trivially yes when $n = 2$?}
%The answer to Question~\ref{quest-open} is evidently ``yes'' when $n = 2$ and ``no'' when $n =3$. In this section, we consider the first nontrivial case, where $n = 4$, and we give a positive answer to Question~\ref{quest-open} by constructing an example of two binary quartic forms $f$ and $g$ such that $R_f$ and $R_g$ are isomorphic yet $I_f$ and $I_g$ do not lie in the same ideal class.

\begin{thm} \label{thm-example}
Let
$$f(X) = 4X^4 - X^3 - 62X^2 + 13X + 255 \quad \text{and} \quad g(X) = 5X^4 - X^3 - 2X^2 - 7X - 6.$$
Then $R_f$ and $R_g$ are isomorphic, but there is no isomorphism between $R_f$ and $R_g$ under which the ideal classes of $I_f$ and $I_g$ are identified. In fact, $I_f$ is principal, whereas $I_g$ is not.
\end{thm}
\begin{proof}
We first prove that $R_f$ and $R_g$ are isomorphic. Consider the map $\iota \colon \on{Sym}_4 \Z^2 \to \Z^2 \otimes_{\Z} \on{Sym}^2 \Z^3$ sending a binary quartic form $f$ as in~\eqref{eq-form} to the pair
$$(A_0, B_f) \defeq \left(\left(\begin{array}{ccc} 0 & 0 & \frac{1}{2} \\ 0 & -1 & 0 \\ \frac{1}{2} & 0 & 0\end{array}\right),\left(\begin{array}{ccc} f_0 & \frac{f_1}{2} & 0 \\ \frac{f_1}{2} & f_2 & \frac{f_3}{2} \\ 0 & \frac{f_3}{2} & f_4 \end{array}\right)\right)$$
The group $\on{GL}_3(\Z) \times \on{GL}_2(\Z)$ acts on the space $\Z^2 \otimes_{\Z} \on{Sym}^2 \Z^3$ via
$$\left(\gamma, \bigl(\begin{smallmatrix} r & s \\ t & u\end{smallmatrix}\bigr)\right) \cdot (A,B) = (r \times \gamma A \gamma^T + s \times \gamma B \gamma^T, t \times \gamma A \gamma^T + u \times \gamma B \gamma^T),$$
so we may think of the polynomial $f$ as giving rise to an orbit of $\on{GL}_3(\Z) \times \on{GL}_2(\Z)$ on $\Z^2 \otimes_{\Z} \on{Sym}^2 \Z^3$ via the map $\iota$. We then have the following result, which translates properties about invariant orders of univariate quartic polynomials into properties about the corresponding orbits of $\on{GL}_3(\Z) \times \on{GL}_2(\Z)$ on $\Z^2 \otimes_{\Z} \on{Sym}^2 \Z^3$:
\begin{prop} \label{prop-param}
Let $f,g \in \Z[X]$ be quartic polynomials such that $\iota(f)$ and $\iota(g)$ are equivalent under the action of $\on{GL}_3(\Z) \times \on{GL}_2(\Z)$. Then $R_f$ and $R_g$ are isomorphic. If $f$ and $g$ are $\on{GL}_2(\Z)$-equivalent, then $\iota(f)$ and $\iota(g)$ are equivalent under the action of $\on{GL}_3(\Z) \subset \on{GL}_3(\Z) \times \on{GL}_2(\Z)$.% The form $f$ is reducible over $\Q$ if and only if the ternary quadratic forms $A_0$ and $B_f$ have a common zero over $\Q$.
\end{prop}
\begin{proof}
This is an immediate consequence of the parametrization of quartic rings given in~\cite[Theorem~2]{MR2113024} together with~\cite[Lemma~2.2]{MR2899953}, which explains how the invariant orders of univariate quartic polynomials fit into this parametrization.
\end{proof}

\noindent To prove that $R_f$ and $R_g$ are isomorphic, it suffices by Proposition~\ref{prop-param} to exhibit the pair $\left(\gamma, \bigl(\begin{smallmatrix} r & s \\ t & u \end{smallmatrix}\bigr)\right) \in \on{GL}_3(\Z) \times \on{GL}_2(\Z)$ such that $\left(\gamma, \bigl(\begin{smallmatrix} r & s \\ t & u \end{smallmatrix}\bigr)\right) \cdot (A_0, B_g) = (A_0, B_f)$. A calculation reveals that taking
$$ \gamma = \left(\begin{array}{ccc} 0 & 2 & -1 \\ -1 & 0 & 1 \\ -3 & -15 & 10 \end{array}\right) \quad \text{and} \quad \left(\begin{array}{cc} r & s \\ t & u \end{array}\right) = \left(\begin{array}{cc} 0 & 1 \\ -1 & 63 \end{array}\right)$$
does the job. Upon observing that $D(f) = D(R_f) = D(R_g) = D(g)$ is squarefree, which implies that $R_f \simeq R_g$ is the maximal order in its field of fractions, verifying that $I_f$ is principal and that $I_g$ is not can be achieved in {\tt sage} using the following code:
\begin{alltt}
R.<x> = PolynomialRing(QQ)
K.<a> = NumberField(4*x^4-x^3-62*x^2+13*x+255)
K.ideal(1,a,4*a^2-a,4*a^3-a^2-62*a).is_principal(proof = True)
L.<b> = NumberField(5*x^4-x^3-2*x^2-7*x-6)
L.ideal(1,b,5*b^2-b,5*b^3-b^2-2*b).is_principal(proof = True)
\end{alltt}
In fact, one can use {\tt sage} to verify that the class group of $R_f = R_g$ is isomorphic to $\Z/2\Z$, so $I_g$ represents the nontrivial class, which squares to the class of $I_f$. This completes the proof of Theorem~\ref{thm-example}.
\end{proof}

We now briefly explain how to search for examples such as the one presented in Theorem~\ref{thm-example}. For simplicity, we restrict our search to irreducible quartic polynomials $f \in \Z[X]$ of squarefree discriminant (so that, in particular, $R_f$ is maximal order in its field of fractions $K_f$). We claim that we can impose the following condition without loss of generality:

\medskip
\noindent \emph{Property $(i)$: The fractional ideal $I_f$ is not principal, and the ideal class group of $R_f$ has a nontrivial $2$-torsion element.}
\medskip

\noindent

\noindent To prove the first part of the claim, observe that if the desired form $g$ exists, then at least one of $I_f$ or $I_g$ is not principal. As for the second part, recall from \S\ref{sec-rings} that $I_f(2)$ represents the ideal class of the inverse different of $R_f$. In particular, if $R_f$ is isomorphic to $R_g$ for some integral binary quartic form $g$, then the ideal classes of $I_f$ and $I_g$ square to the same element of the class group of $R_f$. Thus, if $I_f$ and $I_g$ do not lie in the same ideal class, then the class group of $R_f$ has a nontrivial $2$-torsion element, namely the one represented by $I_fI_g^{-1}$. 

Note that it is not \emph{a priori} obvious that a polynomial $f$ satisfying Property (i) exists, but a computer search reveals many examples.

\medskip

Having narrowed our search to polynomials $f$ satisfying Property (i), we now explain how to construct a quartic polynomial $g \in \Z[X]$ such that $R_f$ and $R_g$ are isomorphic but such that $f$ and $g$ are not $\on{GL}_2(\Z)$-equivalent (note that, by Theorem~\ref{thm-everything}(vii), $f$ and $g$ must be $\on{GL}_2(\Z)$-inequivalent for $I_fI_g^{-1}$ to be non-principal). To construct such a $g$, we claim that it suffices to impose the following condition:

\medskip
\noindent \emph{Property $(ii)$: We have that $\det B_f = 1$ and that $B_f$ is isotropic over $\Q$.}
\medskip

\noindent Suppose Property (ii) is satisfied. Then it follows from the classification of integral ternary quadratic forms that there exists a transformation $\gamma \in \on{GL}_3(\Z)$ such that $\gamma B_f \gamma^T = A_0$. Let $B = \gamma A_0 \gamma^T$, and let $b$ denote the row-$1$, column-$3$ entry of $B$. Acting on the pair $\gamma \cdot (A_0, B_f) = (B, A_0)$ via $\bigl(\begin{smallmatrix} 0 & 1 \\ 1 & -b \end{smallmatrix}\bigr) \in \on{GL}_2(\Z)$, we obtain a pair of the shape $(A_0, \iota(g))$, where $g$ is an integral binary quartic form. By Proposition~\ref{prop-param}, $R_f$ and $R_g$ are isomorphic.

Now, if $f$ and $g$ are $\on{GL}_2(\Z)$-equivalent, then the stabilizer of the pair $(A_0, B_f)$ in $\on{GL}_3(\Z) \times \on{GL}_2(\Z)$ would contain a nontrivial element. But this is impossible: the stabilizer of $(A_0, B_f)$ is simply the group of automorphisms of the ring $R_f$, but because $R_f$ has squarefree discriminant, it has no nontrivial automorphisms.
%it was shown in~\cite{MR2113024} that for such an element to exist, at least one of two conditions must be satisfied: (1) $\det(xA_0 - yB_f)$ is reducible over $\Q$, or (2) $A_0$ and $B_f$ have a common zero over $\Q$. Condition (1) fails by our assumption in Property (ii), and condition (2) fails by Proposition~\ref{prop-param} because $f$ is irreducible over $\Q$.
Thus, we have the claim.

\medskip

One can then generate quartic polynomials $f \in \Z[X]$ satisfying Properties (i) and (ii), apply the above procedure to obtain the form $g$, and check whether $I_fI_g^{-1}$ is principal.

\begin{remark}
Let $\mathcal{O}$ be the ring of integers of a number field. It is a well-known result of Hecke that the ideal class of the different of $\mathcal{O}$ is a perfect square (see~\cite[Theorem~176]{MR638719}). In the discussion~\cite{52815}, Emerton asks whether, among all such square roots, there exists a canonical choice. As part of that discussion, the following observation of Wood is mentioned: when $\mathcal{O} = R_f$ for a polynomial $f \in \Z[X]$ of even degree $n$, it is easy to pick out a ``distinguished'' square root, namely the ideal class of $I_f(\frac{n-2}{2})$. Nevertheless, Theorem~\ref{thm-example} implies that at least when $n = 4$, this ``distinguished'' square root is not particularly canonical, because it depends on the choice of form $f$ such that $\mathcal{O} = R_f$.
\end{remark}

\subsection{On $k$-Hermite equivalence}

Given Theorem~\ref{normformthm}, which establishes that $[f]$ is simply the norm form of $I_f(n-1)$, it is natural to define the following family of generalizations of Hermite equivalence:

\begin{defi}
Let $f,g \in \Z[X]$ be primitive polynomials of degree $n$ and nonzero discriminant, and let $k \in \{0, \dots, n-1\}$. Then $f$ and $g$ are  \emph{$k$-Hermite equivalent} if the norm form of $I_f(k)$ with respect to $R_f$ is $\on{GL}_n(\Z)$-equivalent to the norm form of $I_g(k)$ with respect to $R_g$.
\end{defi}

\noindent Thus   the notion of $(n-1)$-Hermite equivalence coincides with Hermite equivalence.
In addition, we see that $k$-Hermite equivalence and $k'$-Hermite equivalence together imply $(k+k')$-Hermite equivalence. The converse is not in general true---see Theorem~\ref{thm-example} for a counterexample with $n = 4$, $k = 1$, and $k' = 2$. It is easy to verify that, as long as $k \not\in\{0,n-2\}$ (in which case $k$-Hermite equivalence simply amounts to having isomorphic invariant orders), every claim made in Theorems~\ref{thm-main}--\ref{thm-main3} holds with ``Hermite equivalence'' replaced by ``$k$-Hermite equivalence,'' where the occurrences of $n-1$ are replaced by $k$.

For primitive polynomials, we can define $k$-Hermite equivalence for any $k \in \Z$. It follows from Points (iii) and (v) of Theorem~\ref{thm-everything} that the notions of $k$-Hermite equivalence and $(k+n-2)$-Hermite equivalence coincide. %We thus obtain the following result:
%\begin{thm}
%Let $f,g \in \Z[X]$ be primitive polynomials of degree $n$ and nonzero discriminant. The set of integers $k \in \Z$ such that $f$ and $g$ are $k$-Hermite equivalent is the preimage under reduction modulo $n-2$ of a subgroup of $\Z/(n-2)\Z$.
%\end{thm}

\section{Finiteness theorems} \label{sec-allthefiniteness}

The purpose of this section is to prove Theorem~\ref{thm-main2}. In the first part of this section (\S\ref{sec-partI}), we recall several results from the literature concerning finiteness for $\on{GL}_2(\Z)$-equivalence (resp., $\Z$-equivalence) classes of polynomials in $\Z[X]$ (resp., monic polynomials in $\Z[X]$), and we observe that, on account of Theorem~\ref{Prop2.1}, all of these results hold with ``$\on{GL}_2(\Z)$-equivalence'' (resp., ``$\Z$-equivalence'') replaced by ``Hermite equivalence.'' In the second part of this section (\S\ref{sec-partII}), we discuss the extent to which Hermite equivalence classes fall apart into $\on{GL}_2(\Z)$-equivalence and $\Z$-equivalence classes.

\subsection{Finiteness for $\on{GL}_2(\Z)$- and $\Z$-equivalence classes} \label{sec-partI}

Lagrange~\cite{L1773} was the first to develop a reduction theory for quadratic polynomials in $\Z[X]$. His theory was made more precise by Gauss~\cite{G1801}. The theories of Lagrange and Gauss imply in an effective way that there are only finitely many $\GL_2(\Z)$-equivalence classes of quadratic polynomials in $\Z[X]$ with a given nonzero discriminant. Hermite~\cite{H1851} proved the same finiteness statement for the $\GL_2(\Z)$-equivalence classes of cubic polynomials in $\Z[X]$. Furthermore, for polynomials of general degree $n$, Hermite obtained a finiteness result for a suitable, less natural invariant $\Psi$ in place of the discriminant; his theory was made more precise by Julia~\cite{J1917}.

For polynomials of larger degree, Birch and Merriman~\cite{BM1972} proved the following result: \begin{thmalpha}[\protect{Birch and Merriman, \cite{BM1972}}]\label{ThmA}
There are only finitely many $\GL_2(\Z)$-equivalence classes of polynomials in $\Z[X]$ of given degree $n \geq 2$ and given discriminant $D \neq 0$.
\end{thmalpha}
\noindent An immediate consequence of Theorem \ref{ThmA} and Theorem~\ref{Prop2.1} is the following theorem of Hermite:
%Hermite~\cite{H1854,H1857} proved that the set of polynomials with integer coefficients of degree $n$ with a given nonzero discriminant splits into finitely many Hermite equivalence classes.
\begin{thmalpha}[\protect{Hermite,~\cite{H1854,H1857}}]\label{ThmB} There are only finitely many Hermite equivalence classes of polynomials in $\Z[X]$ of given degree $n\geq 2 $ and given discriminant $D\not= 0$.
\end{thmalpha}
\noindent Hermite deduced Theorem~\ref{ThmB} from a reduction theory that he developed for decomposable forms, which he in turn derived from what is now considered an elementary reduction theory for positive definite quadratic forms.

As it happens, Birch and Merriman's proof of Theorem~\ref{ThmA} was \emph{ineffective}. 
On the other hand, a consequence of the theory of Hermite~\cite{H1851} and Julia~\cite{J1917} referenced above is that every polynomial $f\in \Z[X]$ of degree $n\geq 4$ is $\GL_2(\Z)$-equivalent to a polynomial $f^*$ whose height $H(f^*)$ (i.e., maximum absolute value of the coefficients) is effectively bounded above in terms of the aforementioned invariant $\Psi(f)$. In~\cite{EGy1991}, Evertse and Gy\H{o}ry finally proved an \emph{effective} version of Theorem~\ref{ThmA}, and in~\cite{EGy2017}, they improved this result and made it completely explicit. This improved result of Evertse and Gy\H{o}ry is stated as follows:

\begin{thmalpha}[\protect{Evertse and Gy\H{o}ry~\cite[Theorem 14.1.1]{EGy2017}}]\label{ThmC}
Let $f\in\Z[X]$ be a polynomial of degree $n\geq 2$ and discriminant $D\neq 0$. Then $f$ is $\GL_2(\Z)$-equivalent to a polynomial $f^*\in\Z[X]$ for which
$$H(f^*)\leq \exp\left\lbrace\left(4^2n^3\right)^{25n^2}|D|^{5n-3}\right\rbrace.$$
%where $H(f^*)$ denotes the height of $f^*$ $($i.e., the maximum absolute value of the coefficients of $f^*${}$)$.
\end{thmalpha}
\noindent In light of Theorem~\ref{Prop2.1}, Theorem~\ref{ThmC} implies a more precise, effective and quantitative variant of Hermite's result in Theorem~\ref{ThmB}. Theorem~\ref{ThmC} also provides a method to effectively determine in principle all polynomials $f \in \Z[X]$ of given degree $n \geq 2$ and given discriminant $D \neq 0$, up to $\GL_2(\Z)$-equivalence.

For monic polynomials there are finiteness results for $\Z$-equivalence, which do not follow directly
from the results on the weaker $\GL_2(\Z )$-equivalence for arbitrary polynomials mentioned above.
In the case $n=3$, Delone ($=\,$Delaunay) \cite{D1930} and Nagell ($=\,$Nagel) \cite{N1930} proved that there are only finitely many $\Z$-equivalence classes of irreducible monic cubic polynomials in $\Z[X]$ with given nonzero discriminant. The first general \emph{effective} result for monic polynomials,  was proved by Gy\H{o}ry~\cite{Gy1973} for monic polynomials of given nonzero discriminant, where the degree need \emph{not} be fixed:
\begin{thmalpha}[Gy\H{o}ry~\cite{Gy1973}]\label{ThmCC}
There are only finitely many $\Zz$-equivalence classes of monic polynomials in $\Zz[X]$ with given discriminant $D \neq 0$, and a full set of representatives of these classes can be effectively determined.
\end{thmalpha}

\noindent We also mention the following theorem, which is an improved version of a quantitative result of Gy\H{o}ry~\cite{Gy0974} on monic polynomials with given degree and given nonzero discriminant.

\begin{thmalpha}[\protect{Evertse and Gy\H{o}ry,~\cite[Theorem 6.6.2]{EGy2017}}]\label{ThmE}
Let $f\in\Z[X]$ be a monic polynomial of degree $n\geq 2$ and discriminant $D\neq 0$. Then $f$ is $\Z$-equivalent to a polynomial $f^*$ for which
$$H(f^*)\leq \exp\left\lbrace n^{20}8^{n^2+19}\left(|D|(\log^*|D|)^n\right)^{n-1}\right\rbrace.$$
\end{thmalpha}
\noindent
In both Theorems \ref{ThmC} and \ref{ThmE}, the degree $n$ of $f$ can also be estimated from above in terms of $|D(f)|$. %see also Theorem 14.1.2 in Evertse and Gy\H{o}ry (2017).

\begin{thmalpha}[\protect{Gy\H{o}ry,~\cite{Gy0974}}]\label{ThmD}
Every polynomial $f\in \Z[X]$ with nonzero discriminant $D$ has degree
$$n\leq 3+2\log|D|/\log3 .$$
%with equality if and only if $f$ is $\GL_2(\Z)$-equivalent to
%$$XY(X+Y)\qquad\mbox{or}\qquad XY(X+Y)(X^2+XY+Y^2).$$
\end{thmalpha}
%\noindent A generalization for decomposable forms in  $\geq2$ variables is given in Gy\H{o}ry (1994).
\noindent Furthermore, in~\cite{Gy0974} it is given when equality holds in Theorem~\ref{ThmD}. For monic polynomials $f\in\Zz[X]$, the upper bound is slightly improved in~\cite{Gy0974} to $2+2\log|D|/\log3$.

Clearly, Theorem \ref{ThmCC} is a consequence of Theorem \ref{ThmE} and the subsequent estimate for the degree of a polynomial in terms of its discriminant. Likewise, Theorems \ref{ThmC} and \ref{ThmD}
imply that the polynomials in $\Z [X]$ of given discriminant lie in only finitely many $\GL_2 (\Z)$-equivalence classes, a full system of representatives of which can be determined effectively.

For generalizations of Theorems~\ref{ThmA},~\ref{ThmC},~\ref{ThmCC}, and~\ref{ThmE} (e.g., for polynomials with $S$-integral coefficients over number fields) we refer respectively to Birch and Merriman~\cite{BM1972}, to Gy\H{o}ry~\cite{Gy1978a,Gy1978b,Gy1998}, and to Evertse and Gy\H{o}ry~\cite{EGy1991,EGy2017}. Theorem \ref{ThmCC} and its consequences, quantitative versions and generalizations provided effective finiteness results for monogeneity and power integral bases of number fields; cf. Gy\H{o}ry~\cite{Gy0974,Gy1976,Gy1978a,Gy1978b,Gy1980,Gy2000} and Evertse and Gy\H{o}ry~\cite{EGy2017}.

Because Hermite equivalence is a weaker notion than $\on{GL}_2(\Z)$-equivalence, which is in turn strictly weaker than $\Z$-equivalence, Theorems~\ref{ThmA},~\ref{ThmC},~\ref{ThmCC}, and~\ref{ThmE} are more precise than Theorem~\ref{ThmB}. Finiteness theorems concerning unit equations played an important role in the proofs of Theorems~\ref{ThmA},~\ref{ThmC},~\ref{ThmCC}, and~\ref{ThmE}, but such finiteness results were not available to Hermite.

\subsection{Comparison of Hermite, $\on{GL}_2(\Z)$-, and  $\Z$-equivalence} \label{sec-partII}

Theorems \ref{ThmA}--\ref{ThmE} imply in particular that any Hermite equivalence class of separable polynomials (resp., separable monic polynomials) in $\Z[X]$ is a union of at most finitely many $\GL_2(\Zz )$-equivalence classes (resp., $\Zz$-equivalence classes). In the next theorem, we have collected some upper bounds for the number of $\GL_2(\Zz )$-equivalence classes (resp., $\Zz$-equivalence classes) going into an Hermite equivalence class, which are easily derived from the existing literature.

\begin{thm}\label{thm3.5}
We have the following eight points:
\begin{itemize}[leftmargin=35pt]
\item[$(i)$] Separable quadratic polynomials in $\Z[X]$ are Hermite equivalent if and only if they are $\on{GL}_2(\Z)$-equivalent.
\item[$(ii)$] Separable monic quadratic polynomials in $\Z[X]$ are Hermite equivalent if and only if they are $\Z$-equivalent.
\item[$(iii)$] Separable cubic polynomials in $\Z[X]$ are Hermite equivalent if and only if they are $\on{GL}_2(\Z)$-equivalent.
\item[$(iv)$] Every Hermite equivalence class of separable monic cubic polynomials in $\Z[X]$ is a union of at most $10$ $\Zz$-equivalence classes.
\item[$(v)$] Every Hermite equivalence class of separable quartic polynomials in $\Z[X]$ is a union of at most $10$ $\on{GL}_2(\Z)$-equivalence classes $($and at most $7$ if the discriminant is sufficiently large$)$.
\item[$(vi)$] Every Hermite equivalence class of separable monic quartic polynomials in $\Z[X]$ is a union of at most $2760$ $\Z$-equivalence classes $($and at most $182$ if the discriminant is sufficiently large$)$.
\item[$(vii)$] Let $n\geq 5$. Then every Hermite equivalence class of separable degree-$n$ polynomials in $\Z[X]$ is a union of at most $2^{5n^2}$ $\GL_2(\Zz )$-equivalence classes.
\item[$(viii)$] Let $n\geq 5$. Then every Hermite equivalence class of separable monic degree-$n$ polynomials in $\Z[X]$ is a union of at most $2^{5n^2}$ $\Zz$-equivalence classes.
\end{itemize}
\end{thm}

\begin{proof}
(i)--(iii). These points respectively follow from Remarks~\ref{rmk-2},~\ref{rmk-7}, and~\ref{rmk-3}.

\medskip
\noindent (iv). By Theorem~\ref{Prop2.1} (and its converse, which holds when $n = 3$), it suffices to show that every $\GL_2(\Zz )$-equivalence class of separable cubic polynomials in $\Z[X]$ is a union of at most $10$ $\Zz$-equivalence classes. Let $f$ be such a cubic, and let $g$ be the translate of $f$ by an element $\bigl(\begin{smallmatrix}a &b\\ c&d\end{smallmatrix}\bigr)\in\GL_2(\Zz )$. Writing $F(X,Y)$ for the homogenization of $f$, we have that $F(a,c) = 1$, so we obtain a map from monic $\on{GL}_2(\Z)$-translates of $f$ to solutions of the cubic Thue equation $F(x,y) = 1$. It is easy to verify that, under this map, two translates $g$ and $g'$ are sent to the same solution if and only if $g$ and $g'$ are $\Z$-equivalent. The result then follows from a theorem of Bennett~\cite{Ben2001}, which states that the equation $F(X,Y) = 1$ has at most $10$ solutions.
%and consider a monic polynomial that is $\GL_2(\Zz )$-equivalent to $f$,
%say $g(X)=\varepsilon (cX+d)^3f\bigl(\frac{aX+b}{cX+d}\bigr)$
%with
%and $\varepsilon\in\{ \pm 1\}$.
%Let $F(X,Y)=Y^3f(\frac{X}{Y})$ be the binary form associated with $f$. Then $F(a,c)=\varepsilon$, and by changing the signs of $a,b,c,d$ if needed, we may assume that $\varepsilon =1$.
%By a result of Bennett (2001) on cubic Thue equations, there are at most $10$ possibilities for the pair $(a,c)$. For each of these pairs $(a,c)$ we  fix a pair of integers $(b',d')$ such that $ad'-b'c=1$; this is possible since $\gcd (a,c)=1$. Let $g'(X)=(cX+d')^3f(\frac{aX+b'}{cX+d'})$. Then $g'$ is determined by $(a,c)$ and thus, we have at most $10$ possibilities for $g'$.
%
%We now explain how the polynomial $g$ we started with above is related to $g'$.
%We have $ad-bc=\eta \in\{ \pm 1\}$, hence $b=\eta (b' +ta)$, $d=\eta ( d' +tc)$ for some $t\in\Zz$. This shows that $g(X)=\eta^3g'(\eta X+t)$ is $\Zz$-equivalent to $g'$. Hence the polynomials that are $\GL_2(\Zz )$-equivalent to $f$ lie in at most $10$ $\Zz$-equivalence classes.

\medskip
\noindent (v). By Bhargava~\cite[Theorem~1.2]{quartics}, if $\OO$ is an order in a quartic number field, then there are at most $10$ $\on{GL}_2(\Z)$-equivalence classes of quartic polynomials $f \in \Z[X]$ such that $\OO = R_f$ (and at most $7$ if $D(\OO) \gg 1$).  While the work~\cite{quartics} treats only the case of irreducible quartic polynomials, it is well-known that the bound only gets better in the reducible case. The result then follows from Corollary~\ref{cor3.2}.

\medskip
\noindent (vi). By Akhtari and Bhargava~\cite[Theorem~1.1]{quartics}, if $\OO$ is an order in a quartic number field, then there are at most $2760$ $\Z$-equivalence classes of elements $\alpha \in \OO$ such that $\OO = \Z[\alpha]$ (and at most $182$ if $D(\OO) \gg 1$). In the reducible case, we get a bound of $10$ from Point (iv). The result then follows from Corollary~\ref{cor3.4}.

\medskip
\noindent (vii). By Evertse and Gy\H{o}ry~\cite[Theorem~17.1.1]{EGy2017}, there are at most $2^{5n^2}$ $\on{GL}_2(\Z)$-equivalence classes of separable degree-$n$ polynomials in $\Z[X]$ having the same invariant order. The result then follows from Corollary~\ref{cor3.2}.
%, we infer that a Hermite equivalence class of polynomials in $\PI (n)$ falls apart into at most $2^{5n^2}$ $\GL_2(\Zz )$-equivalence classes.

\medskip
\noindent (viii). This point follows from Evertse and Gy\H{o}ry~\cite[Theorem~9.1.4]{EGy2017}. In the case that the Hermite equivalence class under consideration consists of irreducible polynomials of degree $n$,~\cite[Theorem~1.1]{E2011} obtained the slightly better bound $2^{4(n+5)(n-2)}$.
%Consider all polynomials $f\in\MI (n)$ in a given Hermite equivalence class.
%By Corollary \ref{cor3.4}
%these polynomials have up to isomorphism the same invariant order, i.e., there is an order $\OO$ in an algebraic number field field $K$ of degree $n$ such that each of the polynomials $f$ under consideration has a zero $\alpha$ with $\Zz [\alpha ]=\OO$.
%By Evertse (2011, Theorem 1.1),
%the set of $\alpha\in \OO$ with $\Zz [\alpha ]=\OO$
%falls apart in at most $2^{5n^2}$ $\Zz$-equivalence classes.
%By Lemma \ref{lem3.0} this implies
%that the polynomials $f$ lie in at most $2^{5n^2}$ $\Zz$-equivalence classes.
\end{proof}

In \S\ref{section4}--\ref{section5}, we give various examples of pairs of polynomials
$(f,g)$ of degree $n\geq 4$ that are Hermite equivalent
but not $\GL_2(\Zz )$-equivalent. On the other hand we conjecture that for $n\geq 5$,
`most' Hermite equivalence classes consist of only one $\GL_2(\Zz )$-equivalence class. To state our conjecture precisely, and for the sake of convenience in the rest of the paper, we introduce the following notation:
\begin{nota*}
For an integer $n \geq 1$, let $\PI (n)$ denote the set of primitive irreducible polynomials in $\Zz [X]$ of degree $n$, and let $\MI (n) \subset \PI(n)$ denote the subset of monic polynomials. For a number field $K$, let $\PI (K) \subset \PI(n)$ denote the subset of polynomials $f$ with $K_f = K$, and let $\MI (K) \subset \MI(n)$ denote the subset of polynomials $f$ with $K_f=K$.
\end{nota*}

\begin{conj}\label{conj3.6}
Let $K$ be a number field of degree $\geq 5$.
Then among the Hermite equivalence classes of polynomials in $\PI (K)$,
there are only finitely many that split into more than one $\GL_2(\Zz )$-equivalence class.
\end{conj}

\noindent Conjecture~\ref{conj3.6} has already been proved if we restrict our consideration to monic polynomials
and impose some condition on the number field $K$.
Indeed, we have the following result, which is a direct consequence of B\'{e}rczes, Evertse, and Gy\H{o}ry~\cite[Theorem 1.2 (iii)]{BEGy2013}:

\begin{thmalpha}\label{ThmF}
Let $K$ be a number field of degree $n\geq 5$, whose normal closure has Galois group $S_n$. Then among the Hermite equivalence classes of polynomials in $\MI (K)$,
there are only finitely many that split into more than one $\GL_2(\Zz )$-equivalence class.
\end{thmalpha}

\noindent
We note that the method of proof of Theorem~\ref{ThmF} used by B\'{e}rczes et al. is ineffective---i.e., it does not provide a method to compute the exceptional Hermite equivalence classes.

\section{Examples} \label{sec-alltheexamples}

In this section, we prove Theorem~\ref{thm-main3}(ii)--(v) by constructing the relevant infinite collections of polynomials that are Hermite equivalent but not $\on{GL}_2(\Z)$-equivalent.

\subsection{Infinite sequence of monic examples in degree $4$} \label{sec-infmon4}
Theorem \ref{ThmF}, and hence Conjecture \ref{conj3.6}, is false for $n=4$. Indeed, consider the polynomials $f_{r,s}(X)=(X^2-r)^2-X-s$ with $r,s\in\Zz$ such that $f_{r,s}$ is irreducible and the Galois group of the splitting field of $f_{r,s}$ is $S_4$. Let $K_{r,s}$ be the field generated by a zero of $f_{r,s}$. Kappe and Warren~\cite{KW1989} showed that such pairs $(r,s)$ exist, and that there are infinitely many distinct ones among the fields $K_{r,s}$.
B\'{e}rczes, Evertse, and Gy\H{o}ry~\cite{BEGy2013} showed that every field $K_{r,s}$ as above has the following properties:
\begin{itemize}[leftmargin=25pt]
\item[(i)] There are infinitely many pairs of algebraic integers $(\alpha_m ,\beta_m)$ ($m=1,2,\ldots$)
in $K_{r,s}$ such that $\Qq (\alpha_m)=\Qq (\beta_m)=K_{r,s}$, $\beta_m=\alpha_m^2+r_m$, $\alpha_m =\beta_m^2+s_m$ for certain $r_m,s_m\in\Zz$; and
\item[(ii)] There are infinitely many distinct orders among the $\Zz [\alpha_m]$ ($m=1,2,\ldots$).
\end{itemize}
Let $f_m$ be the (monic integral) minimal polynomial of $\alpha_m$ and $g_m$ that of $\beta_m$. Then we have the following result on Hermite equivalence of the polynomials $f_m$ and $g_m$:
\begin{thm} \label{thm-fmandgm}
The polynomials $f_m$ lie in infinitely many distinct Hermite equivalence classes. Moreover, for each $m$, the polynomials $f_m$ and $g_m$ are Hermite equivalent but not $\on{GL}_2(\Z)$-equivalent.
\end{thm}
\begin{proof}
The first claim follows from Point (ii) above in conjunction with Corollary~\ref{cor3.2}. As for the second claim, Point (i) above implies that $\Zz [\alpha_m]=\Zz [\beta_m]$, so by Corollary~\ref{cor3.4}, the polynomials $f_m$ and $g_m$
are Hermite equivalent for each $m$.
If $g_m$ is the translate of $f_m$ by $\bigl(\begin{smallmatrix}a&b\\ c&d \end{smallmatrix}\bigr) \in\GL_2(\Zz )$, then there exists a conjugate $\beta_m'$ of $\beta_m$ such that $\beta_m' = \medfrac{a\alpha_m+b}{c\alpha_m+d} \in \Q(\alpha_m) = K_{r,s} = \Q(\beta_m)$. But since the normal closure of $\Q(\beta_m)$ has Galois group
$S_4$, we must have that $\beta_m' = \beta_m$. Consequently,
\[
\alpha_m^2+r_m=\beta_m=\frac{a\alpha_m+b}{c\alpha_m+d} \Longrightarrow
(c\alpha +d)(\alpha_m^2+r_m)-a\alpha_m-b=0,\] but this is impossible because $\alpha_m$ is of degree $4$ and $a,b,c,d$ are not all
equal to $0$.
\end{proof}

\noindent The argument in the proof of Theorem~\ref{thm-fmandgm} above can be used to produce other pairs $(f,g)$ of primitive irreducible polynomials that are Hermite equivalent but not $\GL_2(\Zz )$-equivalent.
In the next subsection, we construct such pairs of polynomials in degrees $4$ and $5$. The examples we construct are in fact non-monic, unlike the example treated in Theorem~\ref{thm-fmandgm}.

\subsection{Infinite sequences of non-monic examples in degrees $4$ and $5$} \label{sec-infnonmon45}
We start with polynomials of degree $4$. Let $s,t\in\Z$ be such that $s\equiv 1\pmod{15}$ and 
$t\equiv 21\pmod{30}$, and let
$$f(X)=2X^4+8tX^2+2sX-2s^2+8t^2+t,$$
and observe that $f(X)\equiv 2X^4+2X+1\pmod{3}$, that $f(X)\equiv 2(X+1)(X+3)(X^2+X+2)\pmod{5}$ and that $f$ is primitive. These observations imply that $f$ is irreducible in $\Z[X]$ and that the Galois group of $f$ (as a subgroup of $S_4$) contains a transposition and a $4$-cycle and is thus $S_4$.

Now, let $\alpha$ be a zero of $f$, and let $\beta=\alpha+2\alpha^2$.
Then the minimal polynomial of $\beta$ is
\begin{align*}
g(X)= 2X^4&+32X^3t+(-16s^2+192t^2+12s+16t)X^2
\\
&
+(-128s^2t+512t^3-32s^2+32st+128t^2+2s+8t)X
\\
&
+(2s^2-8t^2-t)(16s^2-64t^2+8s-24t-1).
\end{align*}
A computation shows that
\[
\left(\begin{array}{c}
\beta^3\\
\beta^2\\
\beta\\
1
\end{array}\right)=U
\cdot
\left(\begin{array}{c}
\alpha^3\\
\alpha^2\\
\alpha\\
1
\end{array}\right),
\]
where
$$\tiny U=\left(\begin{array}{rrrr}
1-8s-48t & 8s^2+96t^2-12s-28t & 12s^2+32st-48t^2-6s-6t & -32s^2t+128t^3+6s^2-8t^2-3t\\
4 & -16t+1 & -4s & 4s^2-16t^2-2t\\
0 & 2 & 1 & 0\\
0 & 0 & 0 & 1
\end{array}\right).$$
For any $s,t\in\Z$, we have $\det U = 1$; i.e., $U\in \GL_4(\Z)$.
This means that $I_f(3) = I_g(3)$, so by Theorem~\ref{thm-equivalence}, the polynomials $f$ and $g$ are Hermite equivalent. One readily verifies using the argument at the end of the proof of Theorem~\ref{thm-fmandgm} that $f$ and $g$ are not $\on{GL}_2(\Z)$-equivalent.
%
%If $f,g$ were $\GL_2(\Zz )$-equivalent, then by Lemma \ref{lem3.0} there would exist a conjugate $\beta '$ of $\beta$, and $\medmatrix{a&b\\c&d}\in\GL_2(\Zz )$
%such that
%\[
%\beta' =\frac{a\alpha +b}{c\alpha +d} .
%\]
%However, since the Galois group of $f$ is $S_4$ we have $\beta'\not\in\Qq (\alpha )$, unless $\beta' =\beta$. So we have in fact
%\[
%\alpha +2\alpha^2= \beta = \frac{a\alpha +b}{c\alpha +d},
%\]
%that is, $(c\alpha +d)(\alpha +2\alpha^2)-a\alpha -b=0$. But this is impossible,
%since $\alpha$ has degree $4$ and $a,b,c,d$ are not all equal to $0$.
%\\[0.2cm]

We next consider polynomials of degree $5$. Let $s\in\Z$ be such that $s\equiv 71\pmod{110}$, and let
$$f(X)=2X^5+(-800s^2-278s-24)X+800s^2+253s+20.$$
Then observe that
\(
f(X)\equiv 2X^5+3X+3\pmod 5\), that \(f(X)\equiv 2X(X+8)(X+3)(X^2+9)\pmod{11}
\), and that $f$ is primitive. These observations imply that $f$ is irreducible in $\Z[X]$ and that the Galois group of $f$ (as a subgroup of $S_5$) contains a transposition and a $5$-cycle and is thus $S_5$.

Now, let $\alpha$ be a zero of $f$, and let $\beta=\alpha+2\alpha^2$.
Then the minimal polynomial of $\beta$ is
\begin{align*}
g(X) =2X^5&-32(16s+3)(25s+4)X^3+4(25s+4)(96s+13)X^2
\\
&+4(25s+4)(51200s^3+27392s^2+4944s+299)X
\\
& -(32s+5)(25s+4)(19200s^2+6272s+511).
\end{align*}
A computation shows that
\[\left(\begin{array}{c}
\beta^4\\
\beta^3\\
\beta^2\\
\beta\\
1
\end{array}\right)=U
\cdot
\left(\begin{array}{c}
\alpha^4\\
\alpha^3\\
\alpha^2\\
\alpha\\
1
\end{array}\right),\]
where
\[\tiny U=\left(\begin{array}{rrrrr}
6400s^2+2224s+193 & 6400s^2+2424s+224 & -3200s^2-712s-32 &  -6400s^2-1924s-144 & -3200s^2-1012s-80\\
6 & 1 & 3200s^2+1112s+96 & 1600s^2+656s+64 & -4800s^2-1518s-120\\
4 & 4 & 1 & 0 & 0\\
0 & 0 & 2 & 1 & 0\\
0 & 0 & 0 & 0 & 1
\end{array}\right).\]
For any $s\in\Z$, we have $\det U = 1$; i.e., $U\in \GL_5(\Z)$.
This means that $I_f(4) = I_g(4)$, so by Theorem~\ref{thm-equivalence}, the polynomials $f$ and $g$ are Hermite equivalent. Again, one readily verifies using the argument at the end of the proof of Theorem~\ref{thm-fmandgm} that $f$ and $g$ are not $\on{GL}_2(\Z)$-equivalent.
%
%We now proceed completely similarly as above.
%If $f,g$ were $\GL_2(\Zz )$-equivalent, then by Lemma \ref{lem3.0} there would exist a conjugate $\beta '$ of $\beta$, and $\medmatrix{a&b\\c&d}\in\GL_2(\Zz )$
%such that
%\[
%\beta' =\frac{a\alpha +b}{c\alpha +d} .
%\]
%Since the Galois group of $f$ is $S_5$ we have again $\beta' =\beta$, so
%\[
%\alpha +2\alpha^2= \beta = \frac{a\alpha +b}{c\alpha +d},
%\]
%that is, $(c\alpha +d)(\alpha +2\alpha^2)-a\alpha -b=0$, and again we arrive
%at a contradiction,
%since $\alpha$ has degree $5$ and $a,b,c,d$ are not all equal to $0$.
%

\subsection{Further monic examples in degrees $4$, $5$, and $6$}\label{section4}
We start by describing a general strategy by which one can construct examples of Hermite equivalence classes of polynomials that split into multiple $\on{GL}_2(\Z)$-equivalence or $\Z$-equivalence classes. For this, we require the following notation:
\begin{nota*}
Given an algebraic number $\alpha$, we denote by $f_{\alpha}\in\Zz [X]$
the primitive irreducible polynomial with positive leading coefficient
having $\alpha$ as a zero.
\end{nota*}
%From Lemma \ref{lem3.0} we infer the following:
%\\[0.2cm]
%if $\alpha$, $\beta$ are algebraic integers then $f_{\alpha}$, $f_{\beta}$ are $\Zz$-equivalent if and only if $\alpha$ is $\Zz$-equivalent to a conjugate of $\beta$;
%\\[0.15cm]
%if $\alpha$, $\beta$ are algebraic numbers then $f_{\alpha}$, $f_{\beta}$ are $\GL_2(\Zz )$-equivalent if and only if $\alpha$ is $\GL_2(\Zz )$-equivalent to a conjugate of $\beta$.
%\\[0.2cm]
%As argued in the proof of Theorem~\ref{thm-fmandgm}, if the Galois group of the splitting field of $f_{\alpha}$ is the full symmetric group, then the above two assertions hold with $\alpha$ $\Zz$-equivalent/$\GL_2(\Zz )$-equivalent to $\beta$ instead of to a conjugate of $\beta$.
\noindent Recall that a number field $K$ is called \emph{monogenic} if its ring of integers $\OO_K$ can be expressed as $\OO_K = \Zz [\alpha ] = R_{f_\alpha}$. In this case, letting $n=[K:\Qq ]$, the elements
$1,\alpha\kdots\alpha^{n-1}$ form a $\Zz$-module basis of $\OO_K$, and we call this basis a \emph{power integral basis} of $K$. The elements $\alpha\in K$ with $\OO_K =\Zz [\alpha ]$ are precisely those
of discriminant $D(\OO_K)$.
By Corollary \ref{cor3.4}, the minimal polynomials $f_\alpha$ of these
elements $\alpha$ form a Hermite equivalence class.

Let $\alpha$ be an algebraic integer of degree $n\geq4$, and let $f_\alpha(X)=X^n+a_1X^{n-1}+\cdots+a_n\in\Z[X]$ be its minimal polynomial. Consider an element $\beta$ of $\Z[\alpha]$ such that $\Z[\alpha]=\Z[\beta]$. We want to decide whether $\alpha$ and $\beta$ are $\GL_2(\Z)$-equivalent, in the sense that $\beta =\medfrac{a\alpha +b}{c\alpha +d}$ for some
$\bigl(\begin{smallmatrix} a&b\\ c&d \end{smallmatrix}\bigr)\in\GL_2(\Zz )$.\footnote{Observe that two primitive irreducible polynomials $f,g \in \Z[X]$ are $\GL_2(\Zz )$-equivalent if and only if there are a zero $\alpha$ of $f$ and a zero $\beta$ of $g$
such that $\alpha$ and $\beta$ are $\GL_2(\Zz )$-equivalent.} To do this, we write $\beta$ in the form
\begin{equation}\label{3.1}
\beta=b_1+b_2\alpha+\cdots+b_n\alpha^{n-1} \quad \text{with} \quad b_1,\ldots,b_n\in\Z.
\end{equation}
Since $\Zz [\beta -b_1]=\Zz [\beta ]$ and $\beta -b_1$ is $\GL_2(\Zz )$-equivalent to $\alpha$ if and only if $\beta$ is,
we may assume that $b_1=0$. Then $\beta$ is $\GL_2(\Zz )$-equivalent to $\alpha$
if and only if
\begin{equation}\label{3.2}
(c\alpha+d)(b_2\alpha+\cdots+b_n\alpha^{n-1})-(a\alpha+b)=0
\end{equation}
for some $a,b,c,d\in\Z$ with $ad-bc=\pm1$. Representing the left-hand side of (\ref{3.2}) as a linear combination of $1,\alpha,\ldots ,\alpha^n$ and substituting in the relation $\alpha^n=-(a_1\alpha^{n-1}+\cdots+a_n)$, the coefficients of $1,\alpha,\ldots\alpha^{n-1}$ in the resulting expression must all be $0$. We therefore obtain the following system of linear equations in $a,b,c,d$:
\begin{equation}\label{3.3}
\begin{array}{lllllll}
&&&-&cb_na_n & = & b,\\
&& db_2& -&cb_na_{n-1} & = & a, \\
cb_2& +&db_3& -&cb_na_{n-2} & = & 0,\\
cb_3& +&db_4& -&cb_na_{n-3} & = & 0,\\
\vdots&&\vdots&&\vdots&&\vdots\\
cb_{n-2}& +&db_{n-1}& -&cb_na_2 & = & 0,\\
cb_{n-1}& +&db_n& -&cb_na_1 & = & 0.
\end{array}
\end{equation}
We conclude that $\beta$ is $\on{GL}_2(\Z)$-equivalent to $\alpha$ if and only if the system~\eqref{3.3} has a nonzero solution $(a,b,c,d)$ with $ad-bc=\pm 1$.

To determine how the Hermite equivalence class of $f_\alpha$ falls apart into $\on{GL}_2(\Z)$-equivalence classes, it now remains to determine the $\Z$-equivalence classes of elements $\beta \in \Z[\alpha]$ such that $\Z[\alpha] = \Z[\beta]$. This amounts to determining all power integral bases of $\Z[\alpha]$. As it happens, all such bases are explicitly known in several number fields $K$ of degrees $n=4$, $5$, and $6$. Owing to this fact, we can determine how the Hermite equivalence class of polynomials $f\in\MI (K)$ with $R_f = \OO_K$ splits into $\Zz$-equivalence classes and into $\GL_2(\Zz )$-equivalence classes. In the rest of this subsection, we will illustrate this using three concrete examples in degrees $4$, $5$, and $6$.

We start with a polynomial of degree $4$. Let $\alpha$ be a zero of the irreducible polynomial
$$f(X)=X^4-X^3-4X^2+2X+1.$$
Then $D(f)=3981$, which is squarefree, so $R_f$ is the maximal order in $K_f$. A full set of pairwise $\Z$-inequivalent $\beta$ with $R_f=\Z[\beta]$ is given by $\beta=b_2\alpha+b_3\alpha^2+b_4\alpha^3$, where $(b_2,b_3,b_4)$ are listed in Table~\ref{tab1} below; see Ga\'{a}l~\cite[p.~300, last line]{G2019}. Then, by solving the system of linear equations \eqref{3.3} for all pairs $\beta_i,\beta_j$, with $i,j=1,\ldots,10$, we obtain that $\{ \beta_1\kdots\beta_{10}\}$
splits into 3 $\GL_2(\Z)$-equivalence classes:
\[
\{\beta_1,\beta_5,\beta_8\},\ \
\{\beta_2,\beta_6,\beta_7,\beta_{10}\},\ \
\{\beta_3,\beta_4,\beta_9\}.
\]
%\end{minipage}
%\hfill
%\begin{minipage}{6cm}
\vspace*{-5pt}
\begin{figure}[!htbp]
\begin{center}
$$\begin{array}{|c||c|c|c|}
\hline
& b_2& b_3 & b_4\\\hline\hline
\beta_1 & -4 & 0 & 1\\\hline
\beta_2 & -2 & 1 & 0\\\hline
\beta_3 & -1 & 2 & 0\\\hline
\beta_4 & 0 & -1 & 1\\\hline
\beta_5 & 1 & 0 & 0\\\hline
\beta_6 & 1 & 1 & 0\\\hline
\beta_7 & 3 & 1 & -1\\\hline
\beta_8 & 4 & 1 & -1\\\hline
\beta_9 & 15 & 4 & -4\\\hline
\beta_{10} & 21 & 1 & -5\\\hline
\end{array}$$
\end{center}
\caption{The set of $\beta$ with $\Zz [\beta ]=\OO_K$ is the union
of the $\Zz$-equivalence classes represented by the 10 $\Z$-inequivalent elements
\(
\{\beta_1,\dots,\beta_{10}\},
\) with $\beta_5 = \alpha$.}
\label{tab1}
\end{figure}
%\end{minipage}
%\noindent 
Since the Galois group of $f$ is $S_4$, this implies that the Hermite
equivalence class of $f$ splits into $10$ $\Zz$-equivalence classes, represented by $f_{\beta_1}\kdots
f_{\beta_{10}}$, and $3$ $\GL_2(\Zz )$-equivalence classes, represented by
$f_{\beta_i}$ for $i = 1,2,3$.

We next consider a polynomial of degree $5$.
Let $\alpha$ be a zero of the irreducible polynomial
$$f(X)=X^5-5X^3+X^2+3X-1.$$
Then $D(f)=24217$, which is squarefree, so $R_f$ is the maximal order in $K_f$. A full set of pairwise $\Zz$-inequivalent $\beta$ with $R_f=\Zz [\beta]$ is given by $\beta=b_2\alpha+b_3\alpha^2+b_4\alpha^3+b_5\alpha^4$, where $(b_2,b_3,b_4,b_5)$ are listed in Table~\ref{tab2}; see Ga\'{a}l and Gy\H{o}ry~\cite[Example 1]{GGy1999}.

\begin{figure}[!htbp]
\begin{center}
\noindent$\begin{array}{|c||c|c|c|c|}
\hline
& b_2& b_3 & b_4 & b_5\\\hline\hline
\beta_{1} & 0 & 1 & 0 & 0\\\hline
\beta_{2} & 0 & 2 & 1 & -1\\\hline
\beta_{3} & 0 & 4 & 0 & -1\\\hline
\beta_{4} & 0 & 5 & 0 & -1\\\hline
\beta_{5} & 1 & -5 & 0 & 1\\\hline
\beta_{6} & 1 & -4 & 0 & 1\\\hline
\beta_{7} & 1 & -1 & 0 & 0\\\hline
\beta_{8} & 1 & 0 & 0 & 0\\\hline
\beta_{9} & 1 & 1 & -2 & -1\\\hline
\beta_{10} & 1 & 4 & 0 & -1\\\hline
\beta_{11} & 2 & -1 & -1 & 0\\\hline
\beta_{12} & 2 & 4 & -1 & -1\\\hline
\beta_{13} & 2 & 9 & -1 & -2\\\hline
\end{array}$\hspace{0.5cm}
$\begin{array}{|c||c|c|c|c|}
\hline
& b_2& b_3 & b_4 & b_5\\\hline\hline
\beta_{14} & 2 & 15 & -1 & -3\\\hline
\beta_{15} & 2 & 10 & -1 & -2\\\hline
\beta_{16} & 3 & 4 & -1 & -1\\\hline
\beta_{17} & 3 & 5 & -1 & -1\\\hline
\beta_{18} & 3 & 9 & -1 & -2\\\hline
\beta_{19} & 3 & 10 & -1 & -2\\\hline
\beta_{20} & 3 & 14 & -1 & -3\\\hline
\beta_{21} & 3 & 18 & -2 & -4\\\hline
\beta_{22} & 4 & -1 & -1 & 0\\\hline
\beta_{23} & 4 & 0 & -1 & 0\\\hline
\beta_{24} & 4 & 5 & -1 & -1\\\hline
\beta_{25} & 4 & 24 & -2 & -5\\\hline
\beta_{26} & 4 & 29 & -2 & -6\\\hline
\end{array}$\hspace{0.5cm}
$\begin{array}{|c||c|c|c|c|}
\hline
& b_2& b_3 & b_4 & b_5\\\hline\hline
\beta_{27} & 5 & -4 & -1 & 1\\\hline
\beta_{28} & 5 & 8 & -2 & -2\\\hline
\beta_{29} & 5 & 33 & -2 & -7\\\hline
\beta_{30} & 7 & 5 & -2 & -1\\\hline
\beta_{31} & 7 & 9 & -2 & -2\\\hline
\beta_{32} & 7 & 14 & -2 & -3\\\hline
\beta_{33} & 9 & 18 & -3 & -4\\\hline
\beta_{34} & 11 & -13 & -2 & 3\\\hline
\beta_{35} & 12 & 27 & -4 & -6\\\hline
\beta_{36} & 17 & 28 & -6 & -6\\\hline
\beta_{37} & 33 & 30 & -51 & -26\\\hline
\beta_{38} & 83 & 170 & -25 & -39\\\hline
\beta_{39} & 124 & 246 & -40 & -55\\\hline
\end{array}$\\
\end{center}
\caption{The set of $\beta$ with $\Zz [\beta ]=\OO_K$ is the union
of the $\Zz$-equivalence classes represented by the $39$ $\Z$-inequivalent elements $\{\beta_1,\ldots,\beta_{39}\}$, with $\beta_8 = \alpha$.}
\label{tab2}
\end{figure}
 Then, by solving the system of linear equations \eqref{3.3} for all pairs $\beta_i,\beta_j$, with $i,j=1,\ldots,10$, we obtain that $\{ \beta_1\kdots\beta_{39}\}$
splits into $10$ $\GL_2(\Z)$-equivalence classes:
\begin{align*}
&\{\beta_1,\beta_6,\beta_{14},\beta_{35}\},\ \
\{\beta_2,\beta_{12},\beta_{17},\beta_{19},\beta_{33},\beta_{37}\},\ \
\{\beta_3,\beta_{13},\beta_{25},\beta_{31}\},
\\
&\{\beta_4,\beta_{15},\beta_{18},\beta_{23}\},\ \
\{\beta_5,\beta_8,\beta_9,\beta_{16},\beta_{27}\},\ \
\{\beta_7,\beta_{11},\beta_{22},\beta_{39}\},
\\
&\{\beta_{10},\beta_{24},\beta_{26},\beta_{32}\},\ \
\{\beta_{20},\beta_{28},\beta_{29},\beta_{34}\},\ \
\{\beta_{21},\beta_{30}\},\ \
\{\beta_{36},\beta_{38}\}.
\end{align*}
Since the Galois group of $f$ is $S_5$, this implies that the Hermite
equivalence class of $f$ splits into $39$ $\Zz$-equivalence classes, represented by $f_{\beta_1}\kdots
f_{\beta_{39}}$, and $10$ $\GL_2(\Zz )$-equivalence classes, represented by
$f_{\beta_i}$ for $i=1,2,3,4,5,7,10,20,21,36$.

We finally consider a polynomial of degree $6$. Let $\alpha$ be a zero of the irreducible polynomial
$$f(X)=X^6-5X^5+2X^4+18X^3-11X^2-19X+1.$$
Then $D(f)=592661$, which is squarefree, so $R_f$ is the maximal order in $K_f$. A full set of pairwise $\Z$-inequivalent $\beta$ with $R_f=\Z[\beta]$ is given by $\beta=b_2\alpha+b_3\alpha^2+b_4\alpha^3+b_5\alpha^4+b_6\alpha^5$, where $(b_2,b_3,b_4,b_5,b_6)$ are listed in Table~\ref{tab3}; see Bilu, Ga\'{a}l, and Gy\H{o}ry~\cite[Example]{BGGy2004}.

\begin{figure}[!htbp]
\begin{center}
\noindent$\begin{array}{|c||c|c|c|c|c|}
\hline
& b_2& b_3 & b_4 & b_5& b_6\\\hline\hline
\beta_{1} & 1 & 0 & 0 & 0 & 0\\\hline
\beta_{2} & -1 & 1 & 0 & 0 & 0\\\hline
\beta_{3} & -2 & -2 & 1 & 0 & 0\\\hline
\beta_{4} & 2 & 7 & -2 & -3 & 1\\\hline
\beta_{5} & 4 & 9 & -3 & -3 & 1\\\hline
\beta_{6} & -4 & 12 & 0 & -4 & 1\\\hline
\beta_{7} & 5 & -1 & -3 & 1 & 0\\\hline
\beta_{8} & -5 & -5 & 4 & 2 & -1\\\hline
\beta_{9} & 5 & 6 & -2 & -3 & 1\\\hline
\beta_{10} & -5 & 9 & 1 & -4 & 1\\\hline
\beta_{11} & 5 & 9 & -3 & -3 & 1\\\hline
\beta_{12} & -6 & 2 & 3 & -1 & 0\\\hline
\beta_{13} & 6 & -5 & -2 & 1 & 0\\\hline
\beta_{14} & 6 & 8 & -3 & -3 & 1\\\hline
\beta_{15} & 7 & 1 & -4 & 1 & 0\\\hline
\beta_{16} & -7 & 6 & 2 & -1 & 0\\\hline
\beta_{17} & -7 & -6 & 5 & 2 & -1\\\hline
\beta_{18} & 8 & 10 & -4 & -3 & 1\\\hline
\beta_{19} & 9 & 10 & -4 & -3 & 1\\\hline
\beta_{20} & 10 & 0 & -4 & 1 & 0\\\hline
\beta_{21} & 10 & 8 & -6 & -2 & 1\\\hline
\beta_{22} & -10 & -17 & 6 & 6 & -2\\\hline
\beta_{23} & 11 & 3 & -8 & 2 & 0\\\hline
\end{array}$\hspace{0.5cm}
$\begin{array}{|c||c|c|c|c|c|}
\hline
& b_2& b_3 & b_4 & b_5 & b_6\\\hline\hline
\beta_{24} & -11 & -7 & 6 & 2 & -1\\\hline
\beta_{25} & -11 & -13 & 7 & 5 & -2\\\hline
\beta_{26} & -11 & 18 & 2 & -5 & 1\\\hline
\beta_{27} & 12 & 7 & -6 & -2 & 1\\\hline
\beta_{28} & -13 & -6 & 6 & 2 & -1\\\hline
\beta_{29} & 13 & 15 & -8 & -5 & 2\\\hline
\beta_{30} & -14 & -14 & 8 & 5 & -2\\\hline
\beta_{31} & 16 & 16 & -9 & -5 & 2\\\hline
\beta_{32} & 17 & 16 & -9 & -5 & 2\\\hline
\beta_{33} & 18 & 11 & -10 & -4 & 2\\\hline
\beta_{34} & 20 & 22 & -11 & -8 & 3\\\hline
\beta_{35} & 21 & -10 & -8 & 6 & -1\\\hline
\beta_{36} & 22 & 24 & -12 & -8 & 3\\\hline
\beta_{37} & 23 & 14 & -12 & -4 & 2\\\hline
\beta_{38} & -26 & -20 & 14 & 7 & -3\\\hline
\beta_{39} & 43 & 45 & -21 & -14 & 5\\\hline
\beta_{40} & -46 & -45 & 26 & 15 & -6\\\hline
\beta_{41} & 108 & 106 & -63 & -36 & 15\\\hline
\beta_{42} & -119 & -118 & 68 & 40 & -16\\\hline
\beta_{43} & 153 & -26 & -126 & 75 & -12\\\hline
\beta_{44} & 173 & 167 & -105 & -58 & 25\\\hline
\beta_{45} & -590 & -585 & 336 & 198 & -79\\\hline
\end{array}$
\end{center}
\caption{The set of $\beta$ with $\Zz [\beta ]=\OO_K$ is the union
of the $\Zz$-equivalence classes represented by the $45$ $\Z$-inequivalent elements $\{\beta_1,\ldots,\beta_{45}\}$, with $\beta_1 = \alpha$.}
\label{tab3}
\end{figure}
\noindent
Then, by solving the system of linear equations (\eqref{3.3} for all pairs $\beta_i,\beta_j$, with $i,j=1,\ldots,45$, we obtain that $\{ \beta_1\kdots\beta_{45}\}$
splits into $11$ $\GL_2(\Z)$-equivalence classes:
\begin{align*}
&\{\beta_{1},\beta_{19},\beta_{26},\beta_{35},\beta_{42}\},\ \
\{\beta_{2},\beta_{14},\beta_{20},\beta_{23},\beta_{30}\},\ \
\{\beta_{3},\beta_{4},\beta_{13},\beta_{40}\},
\\
&\{\beta_{5},\beta_{15},\beta_{18},\beta_{29},\beta_{38}\},\ \
\{\beta_{6},\beta_{21},\beta_{31},\beta_{39},\beta_{44}\},\ \
\{\beta_{7},\beta_{22},\beta_{33}\},
\\
&\{\beta_{8},\beta_{11},\beta_{24},\beta_{27},\beta_{45}\},\ \
\{\beta_{9},\beta_{28},\beta_{37}\},\ \
\{\beta_{10},\beta_{12},\beta_{36}\},
\\
&\{\beta_{16},\beta_{32},\beta_{34},\beta_{41},\beta_{43}\},\ \
\{\beta_{17},\beta_{15}\}.
\end{align*}
Since the Galois group of $f$ is $S_6$, this implies that the Hermite
equivalence class of $f$ splits into $45$ $\Zz$-equivalence classes, represented by $f_{\beta_1}\kdots
f_{\beta_{45}}$, and $11$ $\GL_2(\Zz )$-equivalence classes, represented by
$f_{\beta_i}$ for $i=1,2,3,5,6,7,8,9,10,16,17$.

%\newpage
\subsection{Reducible monic examples in arbitrary degree}

In this subsection, we prove the following theorem, which shows that it is easy to construct examples of reducible monic polynomials that are Hermite equivalent but not $\on{GL}_2(\Z)$-equivalent:

\begin{thm} \label{thm-redpols}
Let $n \geq 3$ be an integer, and let $f \in \MI(n)$ be such that $f(0) = 1$ and such that $f$ has trivial stabilizer in $\on{GL}_2(\Z)$. Then the monic reducible polynomials $Xf(X) \in \Z[X]$ and $X^{n+1}f\big(\frac{1}{X}\big) \in \Z[X]$ are Hermite equivalent but not $\on{GL}_2(\Z)$-equivalent.
\end{thm}
\begin{proof}
Let $g(X) = Xf(X)$ and $h(X) = X^{n+1}f\big(\frac{1}{X}\big)$. To prove that $g$ and $h$ are Hermite equivalent, it suffices by Corollary~\ref{cor3.4} to prove that $R_g$ and $R_h$ are isomorphic. But since both the leading and constant coefficients of $f$ are equal to $1$, we have that $\on{Res}(X, f(X)) = \on{Res}\big(X,X^nf\big(\frac{1}{X}\big)\big) = 1$. It follows that $R_g$ and $R_h$ are both isomorphic to $\Z \times R_f$, as desired.

Now, if $g$ and $h$ are $\on{GL}_2(\Z)$-equivalent, then the fact that $f$ has no rational roots implies that $h$ is the translate of $g$ by an element $\gamma \in \on{GL}_2(\Z)$ such that $\gamma$ sends $X \mapsto X$ and $f(X) \mapsto X^nf\big(\frac{1}{X}\big)$. But any $\gamma$ stabilizing $X$ is upper-triangular, so $\gamma \cdot \bigl(\begin{smallmatrix} 0 & 1 \\ 1 & 0 \end{smallmatrix}\bigr)$ is a nontrivial transformation stabilizing $f$, which is a contradiction.
\end{proof}

%\newpage
\subsection{Irreducible monic and non-monic examples in arbitrary degree}\label{section5}

In the previous subsections, we gave examples of polynomials of degree $4$, $5$, and $6$ that are Hermite equivalent but not $\on{GL}_2(\Z)$-equivalent. In this subsection, we extend this to every degree $\geq 4$. Our result is as follows:

\begin{thm} \label{thm-ftcandgtc}
We have the following two points:
\begin{itemize}[leftmargin=25pt]
    \item[$(i)$] For every integer $n \geq 4$, there exists an infinite collection of Hermite equivalence classes, each containing two polynomials $f,g \in \PI(n)$ that are properly non-monic $($i.e., not $\on{GL}_2(\Z)$-equivalent to monic polynomials$)$ and not $\on{GL}_2(\Z)$-equivalent.
    \item[$(ii)$] For every integer $n \geq 4$, there exists an infinite collection of Hermite equivalence classes, each containing two polynomials $f,g \in \MI(n)$ that are not $\on{GL}_2(\Z)$-equivalent.
\end{itemize}
\end{thm}

\noindent More precisely, we give, for every integer $n\geq 4$,
an infinite parametric family of pairs of polynomials $(f_{t,c}^{(n)},g_{t,c}^{(n)})$ in $\PI (n)$
where $c$ runs through $1$ and an infinite set of primes and $t$ runs through an infinite set of primes,
with the following properties:
\begin{equation}\label{5.prop}
\left\{\begin{split}
&f_{t,c}^{(n)},g_{t,c}^{(n)}\ \text{have leading coefficient $c$ and are properly non-monic if $c>1$;}
\\
&f_{t,c}^{(n)},g_{t,c}^{(n)}\ \text{are Hermite equivalent but not $\GL_2(\Zz )$-equivalent.}
\end{split}\right.
\end{equation}
In fact, the polynomials $f_{t,c}^{(n)}$ and $g_{t,c}^{(n)}$ will be such that
if we fix $n$ and $c$ and let $t\to\infty$ then the absolute values of the
discriminants of $f_{t,c}^{(n)}$ and $g_{t,c}^{(n)}$ tend to $\infty$.
Since Hermite equivalent polynomials have the same discriminant by Corollary~\ref{cor-discs2},
the pairs $(f_{t,c}^{(n)},g_{t,c}^{(n)})$ lie in infinitely many different Hermite equivalence classes.

\subsubsection{Construction of the polynomials $f_{t,c}^{(n)}$ and $g_{t,c}^{(n)}$} Consider the formal power series in $X$,
\begin{equation}\label{c-id}
C(X)=\frac{1-\sqrt{1-4X}}{2X}=(2X)^{-1}\left(1-\sum_{i=0}^{\infty}\binom{1/2}{i}(-4X)^i\right)
=\sum_{i=0}^{\infty} C_iX^i
\end{equation}
where
$$C_i=\frac{1}{i+1}\cdot \binom{2i}{i}.$$
Recall that $C_i$ is the $i^{\mathrm{th}}$ \emph{Catalan number} and is an integer for each $i$ (see, e.g., Stanley~\cite{stanley}). Next, let $n\geq 4$ be an integer, and let $a^{(n)}(X)$ the $(n-2)$-th partial sum of $C(X)$; i.e., let
$$a^{(n)}(X)=\sum_{i=0}^{n-2}C_i\cdot X^{i}\in\mathbb{Z}[X].$$
Since $C(X)$ satisfies the equation
$$X\cdot C(X)^2-C(X)+1=0,$$
the coefficients of $X^k$ $(k=0,\ldots,n-2)$ in $X\cdot (a^{(n)}(X))^2$ and in $a^{(n)}(X)-1$ are the same. Thus, as polynomials in $\mathbb{Z}[X]$, we have that
\begin{equation}\label{xnf}
X^{n-1}\mid X\cdot (a^{(n)}(X))^2-a^{(n)}(X)+1.
\end{equation}
Let
$$b^{(n)}(X) \defeq \frac{X\cdot (a^{(n)}(X))^2- a^{(n)}(X)+1}{X^{n-1}}.$$
By \eqref{xnf}, we have that $b^{(n)}(X)$ is a polynomial in $\Z[X]$ of degree $n-2$.
Then, substituting $X-X^2$ for $X$ in \eqref{c-id} we obtain
$$C(X-X^2)=\frac{1-\sqrt{(1-2X)^2}}{2(X-X^2)}=\frac{1}{1-X}.$$
Using again that the coefficients of $X^k$ $(k=0,\ldots,n-2)$ in $(1-X)\cdot a^{(n)}(X-X^2)$ and in $(1-X)\cdot C(X-X^2)=1$ are the same, we find that
%$$X^{n-1}\mid (1-X)\cdot a^{(n)}(X-X^2)-1$$ and
\begin{equation}\label{xnfx2x}
X^{n-1}\mid (1-X)\cdot a^{(n)}(X-X^2)-1.
\end{equation}
Let
\begin{align*}
h^{(n)}(X) & \defeq \frac{(1-X)\cdot a^{(n)}(X-X^2)-1}{X^{n-1}} \\
\intertext{and}
k^{(n)}(X) & \defeq -h^{(n)}(1-X)=\frac{1-X\cdot a^{(n)}(X-X^2)}{(1-X)^{n-1}}.
\end{align*}
By \eqref{xnfx2x} $h^{(n)}(X)$ and $k^{(n)}(X)$ are polynomials in $\Z[X]$ of degree $n-2$.
It is easy to check that
\begin{align}\label{ah-relation}
(X-X^2)\cdot a^{(n)}(X-X^2)& =X + h^{(n)}(X)\cdot X^n, \quad \text{and}
\\
\label{bhk-relation}
b^{(n)}(X-X^2)& = -h^{(n)}(X)\cdot k^{(n)}(X).
\end{align}
Now, let $c$ be either $1$ or a prime, and let $t$ be a prime different from $c$. Define the
polynomials
\begin{align*}
&\widetilde{f_{t,c}^{(n)}}(X) \defeq X^n+c^{n-1}t\cdot k^{(n)}(X),
\\
&\widetilde{g_{t,c}^{(n)}}(X) \defeq X^n+c^{n-1}t(1-2Xa^{(n)}(X))+(c^{n-1}t)^2\cdot b^{(n)}(X).
\end{align*}
Notice that, by \eqref{ah-relation} and \eqref{bhk-relation}, we have
\begin{align}\label{fg-relation}
\widetilde{g_{t,c}^{(n)}}(X-X^2)=\widetilde{f_{t,c}^{(n)}}(X)\cdot\widetilde{f_{t,c}^{(n)}}(1-X).
\end{align}
Subsequently, define the polynomials
\begin{align}\label{f-definition}
&f_{t,c}^{(n)}(X) \defeq c^{1-n}\widetilde{f_{t,c}^{(n)}}(cX)=cX^n+t\cdot k^{(n)}(cX),
\\
\label{g-definition}
&g_{t,c}^{(n)}(X) \defeq c^{1-n}\widetilde{g_{t,c}^{(n)}}(cX)=cX^n+t(1-2cXa^{(n)}(cX))+c^{n-1}t^2\cdot b^{(n)}(cX).
\end{align}

\subsubsection{Verifying primitivity and irreducibility}
From the definitions it is clear that both $f_{t,c}^{(n)}$, $g_{t,c}^{(n)}$
are polynomials in $\Z[X]$ of degree $n$ and leading coefficient $c$. Next, since $k^{(n)}(0)=1$, the constant term
of $f_{t,c}^{(n)}$ is $t$. So $f_{t,c}^{(n)}$ is primitive. Furthermore, by Eisenstein's criterion applied at the prime $t$, we see that $f_{t,c}^{(n)}$ is irreducible. The constant term of $g_{t,c}^{(n)}$ is $t$ mod $c^{n-1}t^2$, so $g_{t,c}^{(n)}$ is also primitive, and once again, Eisenstein's criterion implies that it is also irreducible.

\begin{lem}\label{alphabeta}
The polynomials $f_{t,c}^{(n)}$ and $g_{t,c}^{(n)}$ are the minimal polynomials of algebraic numbers $\alpha$ and $\beta$, respectively, such that
$\beta = \alpha -c\alpha^2$. Further, we have $p_{t,c}^{(n)}(\beta )=\alpha$, where
\[
p^{(n)}_{t,c}(X) \defeq X\cdot a^{(n)}(cX)-c^{n-2}t\cdot b^{(n)}(cX).
\]
\end{lem}

\begin{proof} 
The polynomials $f_{t,c}^{(n)}$ and $g_{t,c}^{(n)}$ are primitive and irreducible, so they are the minimal polynomials of certain algebraic numbers. Choose a zero $\alpha$ of $f_{t,c}^{(n)}$,
and put $\beta \defeq \alpha -\alpha^2$.
Note that $\widetilde{f_{t,c}^{(n)}}(c\alpha )=0$, so by \eqref{fg-relation}
we have $\widetilde{g_{t,c}^{(n)}}(c\alpha -(c\alpha )^2)=0$. This implies $g_{t,c}^{(n)}(\beta )=0$. 
This proves the first claim.

As for the second claim, we have by \eqref{ah-relation}, \eqref{bhk-relation}, and \eqref{f-definition} that
\[
\begin{split}
p_{t,c}^{(n)}(X-cX^2) &= c^{-1}(cX-(cX)^2)a^{(n)}(cX-(cX)^2)-c^{n-2}tb^{(n)}(cX-(cX)^2)
\\
&=X+c^{n-1}h^{(n)}(cX)X^n+c^{n-2}th^{(n)}(cX)\cdot k^{(n)}(cX)
\\
&=X+c^{n-2}h^{(n)}(cX)f_{t,c}^{(n)}(X),
\end{split}
\]
and by substituting $X=\alpha$ we get $p_{t,c}^{(n)}(\beta )=\alpha$.
\end{proof}
%\noindent In what follows, we show that if we impose some further congruence conditions on $c$ and $t$, then
%$f_{t,c}^{(n)}$ and $g_{t,c}^{(n)}$ satisfy the properties \eqref{5.prop}.
\noindent We note that the discriminants of $f_{t,c}^{(n)}$ and $g_{t,c}^{(n)}$ are polynomials in $t$ and $c$ that for any fixed value of $c$ tend to $\infty$ with $t$.

\subsubsection{Verifying Hermite equivalence}

We next show that $f_{t,c}^{(n)}$ and $g_{t,c}^{(n)}$ are Hermite equivalent. We first
prove a preparatory lemma.
%Recall that if $\gamma$ is an algebraic number of degree $n$, then
%$\MM_{\gamma}$ denotes the $\Zz$-module generated by $1,\gamma\kdots\gamma^{n-1}$.

\begin{lem}\label{fqr}
Take $f\in \PI(n)$ with leading coefficient $c$, and let $\gamma \in K_f$ be a root of $f$. Further, take $s \in \Z[X]$, and let $p(X)=Xs(cX)$.
Then $p(\gamma )^k\in I_f(n-1) = \Z\langle 1, \gamma, \dots, \gamma^{n-1} \rangle$ for each $k=0\kdots n-1$.
\end{lem}
\begin{proof}
Let $\widetilde{f}(X)\defeq  c^{n-1}f(c^{-1}X)$. Then $\widetilde{f}$ is monic and in $\Zz [X]$. 
Let $i\geq n$. Then
there exist polynomials $\widetilde{q},\widetilde{r}\in\Zz [X]$, such that 
$X^i=\widetilde{q}(X)\widetilde{f}(X)+\widetilde{r}(X)$, where the degree of $\widetilde{r}$ is less than $n$. By substituting $cX$ for $X$
and then dividing by $c^{n-1}$ we get that 
there exist $q,r\in\mathbb{Z}[X]$ such that
$c^{i-n+1}X^i=q(X)\cdot f(X)+r(X)$, where the degree of $r$ is less than $n$.
This implies that $c^{i-n+1}\gamma^i\in I_f(n-1)$ for every integer $i\geq n$.

Let $k\in\{ 0\kdots n-1\}$.
Observe that $c^i$ divides the coefficient of $X^i$ in $(c\cdot p(X))^k$. Therefore if $k<n$ and $i\geq n$, then $c^{i-k}$ divides the coefficient of $X^i$ in $p(X)^k$.
It follows that $p(\gamma )^k$ is a $\Zz$-linear combination of $1,\gamma\kdots \gamma^{n-1}$
and $c^{i-n+1}\gamma^i$ for $i\geq n$.
Hence $p(\gamma)^k\in I_f(n-1)$.
\end{proof}

\begin{prop}\label{gln}
Let $n\geq 3$. Then $f_{t,c}^{(n)}$
and $g_{t,c}^{(n)}$ are Hermite equivalent.
\end{prop}

\begin{proof}
For short, we write $f = f_{t,c}^{(n)}$ and $g = g_{t,c}^{(n)}$. By Lemma \ref{alphabeta}, $\beta =\alpha -c\alpha^2$
is a zero of $g$. In view of Theorem~\ref{thm-equivalence}, it suffices to show that
the $\Z$-modules $I_f(n-1) = \Z\langle1, \alpha, \dots, \alpha^{n-1} \rangle$ and $I_g(n-1) = \Z\langle1, \beta, \dots, \beta^{n-1} \rangle$ coincide.

From Lemma \ref{fqr} with $p(X)=X-cX^2=X(1-cX)$, it follows that
$1,\beta\kdots \beta^{n-1}\in I_f(n-1)$, so
$I_g(n-1)\subseteq I_f(n-1)$. For the other direction, we apply Lemma \ref{fqr} with $f$ replaced by $g$ and with $p=p_{t,c}^{(n)}+c^{n-2}tb^{(n)}(0)$. Notice that $p(X)/X$ is of the form $s(cX)$ with $s\in\Zz [X]$.
It follows that $p(\beta )^k\in I_g(n-1)$ for $k=0\kdots n-1$.
By Lemma \ref{alphabeta} we have $p(\beta )=\alpha + c^{n-2}tb^{(n)}(0)$.  Thus, $\alpha^k\in I_g(n-1)$ for $k=0\kdots n-1$, so $I_f(n-1) \subseteq I_g(n-1)$.
\end{proof}

\subsubsection{Verifying (proper non-)monicity}
Clearly, if $c = 1$, then both $f_{t,c}^{(n)}$ and $g_{t,c}^{(n)}$ are monic. We now show that if $c\neq 1$ and $n \geq 4$, then for an appropriate choice of $t$ and $c$,
$f_{t,c}^{(n)}$ and $g_{t,c}^{(n)}$ are properly non-monic.

Take $c$ to be prime with $c\equiv 1\pmod n$, and consider the subgroup of $\Ff_c^*=(\Zz /c\Zz )^*$ given by
$$S_{n,c}=\{ \pm r^n\mid r\in\Ff_c^*\}.$$
Since $c\equiv 1\pmod n$ and $n\geq 4$, the order of $S_{n,c}$ is at most $\frac{2(c-1)}{n}<c-1$ (it contains up to sign all powers of $g^n$, where $g$ is a primitive root modulo $c$),
so it is a proper subgroup of $\Ff_c^*$.

\begin{lem}\label{gl2monic}
Assume that $n\geq 4$, let $c$ be a prime with $c\equiv 1\pmod n$, and let $t$ be a prime
different from $c$ with $t\,({\rm mod}\, c)\not\in S_{n,c}$.
Then $f_{t,c}^{(n)}$ and $g_{t,c}^{(n)}$ are properly non-monic.
\end{lem}

\begin{proof}
Let $\bigl(\begin{smallmatrix}a&b\\d&e\end{smallmatrix}\bigr) \in \GL_2(\Zz )$ be any element, and consider the polynomial
$\pm (dX+e)^nf_{t,c}^{(n)}\bigl(\medfrac{aX+b}{dX+e}\bigr)$ given by translating $f_{t,c}^{(n)}$ by $\bigl(\begin{smallmatrix}a&b\\d&e\end{smallmatrix}\bigr)$ (up to sign). The leading coefficient of this polynomial is $\pm F(a,d)$, where
$F(X,Y)=Y^nf_{t,c}^{(n)}(X/Y)$ is the homogenization of $f$. We have
\[
F(a,d)\equiv t\cdot d^n\not\equiv\pm 1 \pmod c .
\]
Hence $f_{t,c}^{(n)}$ is properly non-monic. The same argument shows that $g_{t,c}^{(n)}$ is properly non-monic.
\end{proof}

\subsubsection{Verifying $\on{GL}_2(\Z)$-inequivalence}
We now prove that, for every integer $n\geq 4$, there exist infinitely many parameters
$c,t$ satisfying the conditions from Lemma \ref{gl2monic} such that $f_{t,c}^{(n)}$ and $g_{t,c}^{(n)}$ are not $\on{GL}_2(\Z)$-equivalent. We need some preparatory lemmas.

\begin{lem}\label{kn-irreducible}
Let $n\geq 4$ be an integer. Then the polynomial $k^{(n)}$ is irreducible.
\end{lem}

\begin{proof}
We checked by computer the irreducibility of $k^{(n)}$ for $4\leq n\leq 38$, so in the
course of the proof we may assume that $n\geq 39$.
The proof consists of two steps: we first show that $k^{(n)}$ is
either irreducible or has a rational root, and second that $k^{(n)}$ cannot have a rational root.

\medskip
\noindent {\sf Step 1:}  In this step, we use an argument based on Newton polygons. To set up this argument, we first show that the polynomials $k^{(n)}$ satisfy the recursive equation
\begin{equation}\label{req}
(X-1)\cdot k^{(n+1)}(X)+k^{(n)}(X)=C_{n-1}\cdot X^n.
\end{equation}
To prove~\eqref{req}, let $n\geq 2$, and recall that $k^{(n)}(X)=-h^{(n)}(1-X)$. It therefore suffices to prove the recursive equation
\begin{equation} \label{req2}
X\cdot h^{(n+1)}(X)-h^{(n)}(X)=C_{n-1}\cdot (1-X)^n.
\end{equation}
By the definition of $h^{(n)}(X)$ we have the following:
\begin{align*}
X\cdot h^{(n+1)}(X)-h^{(n)}(X) &=
X\cdot \frac{(X-X^2)\cdot a^{(n+1)}(X-X^2)-X}{X^{n+1}}
-\frac{(X-X^2)\cdot a^{(n)}(X-X^2)-X}{X^n}
\\
& =
\frac{(X-X^2)\cdot (a^{(n+1)}(X-X^2)-a^{(n)}(X-X^2))}{X^n}
\\
& =
\frac{(X-X^2)\cdot C_{n-1}\cdot(X-X^2)^{n-1}}{X^n}=C_{n-1}\cdot(1-X)^n,
\end{align*}
so~\eqref{req2} holds. Now,
%where $C_{n-1}=\medfrac{1}{n}\medbinom{2n-2}{n-1}$.
let
$$
K^{(n)}(X) \defeq C_{n-1}\cdot\sum_{i=0}^{n-2}\binom{n}{i}
\cdot \frac{(n-1-i)(n-i)}{(n-1+i)(n+i)}\cdot X^i.
$$
We claim that
\begin{equation}\label{kK-relation}
K^{(n)}(X)=k^{(n)}(X+1) \quad \text{for}\quad n\geq 2.
\end{equation}
It is straightforward to check that $K^{(2)}(X)=k^{(2)}(X+1)=1$, $K^{(3)}(X)=k^{(3)}(X+1)=X+2$.
By comparing the coefficients of $X^i$ $(i=0,\ldots,n)$ on the left- and right-hand sides,
we deduce the recursive equation
$$X\cdot K^{(n+1)}(X)+K^{(n)}(X)=C_{n-1}\cdot (X+1)^n\ \ \text{for } n\geq 2.$$
Together with \eqref{req}, this implies \eqref{kK-relation}.

Henceforth we assume that $n\geq 39$. We now show that $k^{(n)}$ is either irreducible or has a rational root.
Of course it suffices to prove this for $K^{(n)}(X)$ instead of $k^{(n)}(X)$. 
We apply a theorem of Dumas (see Dumas~\cite[pp.~236--237]{Dumas1906} or Mott~\cite[Proposition 3.2]{Mott1995}) which for convenience of the reader we recall here. 

For a prime number $l$ and an integer $a$, let $v_l(a)$ denote the largest integer $m$ such that $l^m$ divides $a$. Given a polynomial $f(X)=a_0X^s+a_1X^{s-1}+\cdots +a_s\in\Zz [X]$ and a prime $l$, 
the Newton polygon $N_{f,l}$ is the lower convex hull of the points $(i,v_l(a_i))$ ($i=0\kdots s$).
Let $i_0=0<i_1<\cdots <i_u=s$ be the indices such that $(i_j,v_l(a_{i_j}))$ ($j=0\kdots u$) are
the vertices of $N_{f,l}$ and put $n_j\defeq i_j-i_{j-1}$, $m_j\defeq v_l(a_{i_j})-v_l(a_{i_{j-1}})$,
$d_j\defeq \gcd (m_j,n_j)$, $w_j\defeq n_j/d_j$ for $j=1\kdots u$.  

\begin{prop}[Dumas' Irreducibility Theorem]\label{dumas}
The degree of any non-trivial factor in $\Zz [X]$ of $f(X)$ must be a sum of the form
$\sum_{j=1}^u t_jw_j$, where $t_j\in\Zz$ and $0\leq t_j\leq d_j$, for $j=1\kdots t$.
\end{prop}

Let $p$ and $q$ be prime numbers for which $n<p<\frac{6n}{5}$ and $\frac{6n}{5}< q<\frac{36n}{25}$. By the results of Nagura~\cite{nagura}, such primes exist for $n> 24$.
Let $a_i$ denote the coefficient of $X^i$ in $K^{(n)}$.
It is easy to calculate the $l$-adic valuation $v_l(a_i)$ of $a_i$ for $l\in\{ p,q\}$.
The properties of $p$ and $q$ imply that for $l\in\{ p,q\}$ we have
$$v_l(C_{n-1})=1 \quad \text{and} \quad v_l\left( \binom{n}{i}\cdot (n-1-i)\cdot (n-i)\right) =0 \quad \text{for} \quad i=0\kdots n-2.$$
Thus, for $l\in\{ p,q\}$ we have that $v_l(a_i)=0$ if $n-1+i=l$ or $n+i=l$, and $v_l(a_i)=1$ otherwise.
With these observations, 
one easily verifies that the Newton polygon $N_{K^{(n)},l}$
consists of three edges connecting the points
$$(0,1),\qquad (l-n,0),\qquad (l-n+1,0),\qquad (n-2,1).$$
Proposition \ref{dumas}
now implies that for $l=p$ as well as $l=q$, the following holds:
if $K^{(n)}$ factors  over $\Qq$, then it must have at most three irreducible factors of degrees which are sums of the numbers $l-n$, $1$ and $2n-l-3$. For $l=p$, we obtain the following five possibilities:
\begin{itemize}[leftmargin=15pt]
\item $K^{(n)}(X)$ is irreducible;
\item $K^{(n)}(X)=K_1(X)\cdot K_2(X)$, where $\deg(K_1)=p-n$ and $\deg(K_2)=2n-p-2$;
\item $K^{(n)}(X)=K_1(X)\cdot K_2(X)$, where $\deg(K_1)=p-n+1$ and $\deg(K_2)=2n-p-3$;
\item $K^{(n)}(X)=K_1(X)\cdot K_2(X)$, where $\deg(K_1)=n-3$ and $\deg(K_2)=1$; or
\item $K^{(n)}(X)=K_1(X)\cdot K_2(X)\cdot K_3(X)$, where $\deg(K_1)=p-n$, $\deg(K_2)=1$, and $\deg(K_3)=2n-p-3$.
\end{itemize}
We can conclude something similar for $l=q$, obtaining the following five possibilities:
\begin{itemize}[leftmargin=15pt]
\item $K^{(n)}(X)$ is irreducible;
\item $K^{(n)}(X)=K_1^*(X)\cdot K_2^*(X)$, where $\deg(K_1^*)=q-n$ and $\deg(K_2^*)=2n-q-2$;
\item $K^{(n)}(X)=K_1^*(X)\cdot K_2^*(X)$, where $\deg(K_1^*)=q-n+1$ and $\deg(K_2^*)=2n-q-3$;
\item $K^{(n)}(X)=K_1(X)\cdot K_2(X)$, where $\deg(K_1)=n-3$ and $\deg(K_2)=1$; or
\item $K^{(n)}(X)=K_1^*(X)\cdot K_2^*(X)\cdot K_3^*(X)$, where $\deg(K_1^*)=q-n$, $\deg(K_2^*)=1$, and $\deg(K_3^*)=2n-q-3$.
\end{itemize}

\noindent Since $p<q$ and $p+q<\frac{6n}{5}+\frac{36n}{25}<3n-4$ for $n\geq 39$, we have that
\[
p-n\neq 2n-q-2,\ \
p-n\neq 2n-q-3,\ \
p-n+1\neq 2n-q-2,\ \
p-n+1\neq 2n-q-3.
\]
Therefore, the second and the third possibilities above can not be true, which means that $K^{(n)}(X)$, and so $k^{(n)}(X)$, is either irreducible or has a rational root.

\medskip
\noindent {\sf Step 2:} In this step, we prove that $k^{(n)}$ does not have a rational root; consequently, by the result of Step 1, it is irreducible.

We keep our assumption $n\geq 39$. Applying the recursive equation~\eqref{req} twice in succession, we obtain the following identity relating $k^{(n+2)}$ and $k^{(n)}$:
\begin{equation}\label{doubreq}
(X-1)^2\cdot k^{(n+2)}(X)=k^{(n)}(X)+C_{n}\cdot X^{n}\cdot\left( X^2-X-\frac{C_{n-1}}{C_{n}}\right).
\end{equation}
By the definition of the Catalan numbers, we have that $\frac{C_{n-1}}{C_{n}}=\frac{n+1}{4n-2}$, so the roots of $X^2-X-\frac{C_{n-1}}{C_{n}}$ are given by
$$\alpha_1=\frac{1-\sqrt{2+\frac{3}{2n-1}}}{2}\mbox{ and }\alpha_2=\frac{1+\sqrt{2+\frac{3}{2n-1}}}{2}.$$
Furthermore, since $n\geq 39$, we have that
$$-\frac{1}{4}<\alpha_1<-\frac{1}{5}\mbox{, and }\frac{6}{5}<\alpha_2<\frac{5}{4}.$$
Therefore,
\begin{itemize}[leftmargin=15pt]
\item if $x\leq -\frac{1}{4}$ and $n$ is odd, then
\begin{equation}\label{odd4}
C_{n}\cdot x^{n}\cdot\left( x^2-x-\frac{C_{n-1}}{C_{n}}\right)<0,
\end{equation}
\item if $x\leq -\frac{1}{4}$ and $n$ is even, then
\begin{equation}\label{even4}
C_{n}\cdot x^{n}\cdot\left( x^2-x-\frac{C_{n-1}}{C_{n}}\right)>0,
\end{equation}
\item if $-\frac{1}{5}\leq x<0$ and $n$ is odd then
\begin{equation}\label{odd5}
C_{n}\cdot x^{n}\cdot\left( x^2-x-\frac{C_{n-1}}{C_{n}}\right)>0,
\end{equation}
\item if $x>0$ then
\begin{equation}\label{nulla}
k^{(n)}(x)>0,
\end{equation}
\noindent because the coefficients of $k^{(n)}$ are all positive. This fact about the coefficients of $k^{(n)}$ is easily proven by induction: one simply divides the recursive equation~\eqref{req} through by $X-1$, applies the formal power series expansion $\frac{1}{1-X} = 1 + X + X^2 + \cdots$, and verifies that the coefficients of $k^{(n+1)}$ are positive if the same holds for the coefficients of $k^{(n)}$.
\end{itemize}
It is easy to check by computer that the first derivative of $k^{(39)}$ is strictly positive (e.g., its absolute minimum is strictly positive), so $k^{(39)}$ is monotonically increasing. Since $k^{(39)}(-1/4)<0$ and $k^{(39)}(-1/5)>0$, we have that
\begin{itemize}[leftmargin=15pt]
\item if $x\leq -\frac{1}{4}$, then $k^{(39)}(x)<0$; and
\item if $-\frac{1}{5}\leq x$, then $k^{(39)}(x)>0$.
\end{itemize}
Therefore, by \eqref{doubreq}, \eqref{odd4}, \eqref{odd5} and \eqref{nulla} the same is true for every odd integer $n>39$:
\begin{itemize}[leftmargin=15pt]
\item if $x\leq -\frac{1}{4}$ and $n$ is odd, then $k^{(n)}(x)<0$; and
\item if $-\frac{1}{5}\leq x$ and $n$ is odd, then $k^{(n)}(x)>0$.
\end{itemize}
\noindent On the other hand, it is easy to check that $k^{(4)}(x)$ is strictly positive, so by \eqref{doubreq},\eqref{even4}, and \eqref{nulla}:
\begin{itemize}[leftmargin=15pt]
\item if $x\leq -\frac{1}{4}$ and $n$ is even, then $k^{(n)}(x)>0$; and
\item if $0\leq x$ and $n$ is even, then $k^{(n)}(x)>0$.
\end{itemize}
It remains to handle the interval $-\frac{1}{5}\leq x<0$. For this, we can use the fact that if $n$ is even, then $k^{(n+1)}(x)>0$ and so by \eqref{req}, we have that
$$k^{(n)}(x)=C_{n-1}\cdot x^{n}+(1-x)\cdot k^{(n+1)}(x)>0.$$
It follows that if $n\geq 39$, then all real roots of $k^{(n)}$ are strictly between $-\frac{1}{4}$ and $-\frac{1}{5}$. However, the coefficients of $k^{(n)}$ are all positive integers, and its constant term is 1, so all rational roots are of the form $-\frac{1}{d}$, where $d$ divides the leading coefficient. As there are no such numbers strictly between $-\frac{1}{4}$ and $-\frac{1}{5}$, $k^{(n)}$ has no rational roots. Therefore, $k^{(n)}$ is irreducible for any $n\geq 4$.
\end{proof}

\begin{lem}\label{kn-norootmodp}
Let $n\geq 4$ be an integer. Then there exist infinitely many primes $p$
such that $k^{(n)}$ has no root modulo $p$.
\end{lem}

\begin{proof}
A well-known consequence of
Chebotarev's density theorem or Frobenius' density theorem for primes  asserts that if $f \in \Z[X]$ is irreducible, then there are infinitely many primes $p$ modulo which
$f$ has no roots; see for instance Lenstra and Stevenhagen~\cite{LS1996} and Serre~\cite{Serre2003}.
\end{proof}

\begin{prop}\label{nogl2}
Let $n\geq 4$ be an integer, let $p>C_{n-1}$ be a prime
such that $k^{(n+1)}$ has no root modulo $p$, and let $c$ be either $1$ or a prime and $t$ a prime
such that
\begin{equation}\label{tc-properties}
c\equiv 1 \pmod p \quad \text{and} \quad t\equiv -C_{n-1}^{-1}\pmod p .
\end{equation}
Then $f_{t,c}^{(n)}$ and $g_{t,c}^{(n)}$ are not $\GL_2(\Zz )$-equivalent.
\end{prop}

\begin{proof}
By Lemma \ref{alphabeta} we may write
\begin{equation}\label{gykt}
f_{t,c}^{(n)}(X)= c\prod_{i=1}^{n}(X-\alpha_i),\quad
g_{t,c}^{(n)}(X)=c\prod_{i=1}^n(X-\beta_i),\
\text{where  $\beta_i=\alpha_i-c\alpha_i^2\ \text{for } i=1\kdots n$}.
\end{equation}
Assume that $f_{t,c}^{(n)}$ and $g_{t,c}^{(n)}$ are $\on{GL}_2(\Z)$-equivalent.
Then there is $\bigl(\begin{smallmatrix}a&b\\ d&e\end{smallmatrix}\bigr)\in\GL_2(\Zz )$ such that
\[
\begin{split}
g_{t,c}^{(n)}(X) &=\pm (dX+e)^nf_{t,c}^{(n)}\bigl(\medfrac{aX+b}{dX+e}\bigr)
\\
%&=\pm c\prod_{i=1}^n\big( (aX+b)-\alpha_i(dX+e)\big)
%\\
&=\pm c\prod_{i=1}^n(a-\alpha_id)\cdot\prod_{i=1}^n\bigl(X+\medfrac{b-\alpha_ie}{a-\alpha_id}\bigr).
\end{split}
\]
This implies that $\pm\prod_{i=1}^n(a-\alpha_id)=1$ and that the sets
$$\{\beta_1,\beta_2,\ldots,\beta_n\} \quad \text{and}\quad \left\lbrace-\frac{b-\alpha_1e}{a-\alpha_1d},-\frac{b-\alpha_2e}{a-\alpha_2d},\ldots,-\frac{b-\alpha_ne}{a-\alpha_nd}\right\rbrace,$$
both of which have order $n$, coincide.

We now split into two cases. The first case is that
$$\beta_1=\alpha_1-c\alpha_1^2=-\frac{b-\alpha_1e}{a-\alpha_1d},$$
and the other case is that
$$\beta_1=\alpha_1-c\alpha_1^2=-\frac{b-\alpha_je}{a-\alpha_jd}$$
for some $j \neq 1$. In the first case, we have that
$$-cd\alpha_1^3+(ac+d)\alpha_1^2+(-a+e)\alpha_1-b=0.$$
But $\alpha_1$ is of degree $n>3$, so the coefficients of $1,\alpha_1,\alpha_1^2,\alpha_1^3$ above are $0$. This implies that $a=b=d=e=0$, which is a contradiction, so this case is not possible.

The second case is more complicated. In this case, we have that
\begin{equation}\label{2case}
(\alpha_1-c\alpha_1^2)\cdot(a-\alpha_jd)=(\alpha_je-b),
\end{equation}
for some $j\not=1$.
%Our aim is to prove that this can not be satisfied with $a,b,d,e\in\Z$, $ae-bd=\pm1$.
By using identity \eqref{req} and the conditions \eqref{tc-properties}, we obtain
\begin{equation}\label{fmn}
f^{(n)}_{t,c}(X)\equiv X^n-\frac{1}{C_{n-1}}\cdot k^{(n)}(X)
\equiv\frac{X-1}{C_{n-1}}\cdot k^{(n+1)}(X)\pmod{p}
\end{equation}
So $1$ is a root of $f^{(n)}_{t,c}(X)$ modulo $p$. But then from \eqref{gykt} it follows that
\begin{equation} \label{prod0modp}
\prod_{i = 1}^n (1-\alpha_i)\equiv 0 \pmod p .
\end{equation}
Let $\fp$ be an arbitrary prime ideal divisor of $p$ in the splitting field $L$ of $f_{t,c}^{(n)}$. It follows from~\eqref{prod0modp} that $1-\alpha_k\equiv 0\pmod{\fp}$ for some $k$ with $1\leq k\leq n$. There exists $\varphi\in \on{Gal}(L/\mathbb{Q})$ for which $\varphi(\alpha_k)=\alpha_1$; let $\fp_1 = \varphi(\fp)$ be the corresponding prime ideal of $L$. Then $1-\alpha_1\pmod 0\pmod{\fp_1}$,
 and combining this with \eqref{fmn} and \eqref{gykt} yields that
\begin{equation}\label{k-property}
k^{(n+1)}(X)\equiv C_{n-1}\cdot c\prod_{i = 2}^n (X-\alpha_i)\pmod{\fp_1}.
\end{equation}
However, $c\equiv 1\pmod{\fp_1}$, so by \eqref{2case} we have that
\begin{equation}\label{aeb}
\alpha_je-b\equiv 0\pmod{\fp_1}.
\end{equation}
Now, $\fp_1\nmid e$, for otherwise $\fp_1\mid b$, which would contradict $ae-bd=\pm1$. Thus $p\nmid e$; let $e^{-1}$ denote the inverse of $e$ in $\mathbb{Z}/p\mathbb{Z}$.
Then~\eqref{aeb} gives $\alpha_j\equiv e^{-1}b\pmod{\fp_1}$, and by substituting this into \eqref{k-property} we find that
\(
k^{(n+1)}(e^{-1}b)\equiv 0\pmod{\fp_1}
\). Taking norms with respect to $L/\mathbb{Q}$, we infer that $k^{(n+1)}(e^{-1}b)\equiv 0\pmod{p}$,
contradicting the choice of $p$.
Therefore, the second case is not possible as well, and we conclude that $f_{t,c}^{(n)}$ and $g_{t,c}^{(n)}$ are not $\GL_2(\Zz )$-equivalent.
\end{proof}

Finally, upon combining Lemma \ref{gl2monic}, Proposition \ref{gln}, and Proposition \ref{nogl2},
we obtain the following result:

\begin{thm}\label{mind}
Let $n\geq 4$
be an integer, and let $p$ be a prime with $p>C_{n-1}=\medfrac{1}{n}\cdot\medbinom{2n-2}{n-1}$ such that $k^{(n+1)}$ has no root modulo $p$.
Further, let $c$ be either $1$ or a prime, and let $t$ be any prime with
\[
c\equiv 1\pmod{np}, \quad t\equiv -C_{n-1}^{-1}\pmod {p}, \quad t\not= c,\quad \text{and} \quad t\,\,({\rm mod}\, c)\not\in S_{n,c}.
\]
Then the polynomials $f_{t,c}^{(n)}$ and $g_{t,c}^{(n)}$ given respectively by \eqref{f-definition}
and \eqref{g-definition} lie in $\PI(n)$ and satisfy the following properties:
\begin{itemize}[leftmargin=25pt]
\item[$(i)$] $f_{t,c}^{(n)}$, $g_{t,c}^{(n)}$ have leading coefficient $c$ and are properly non-monic if $c>1$;
\item[$(ii)$] $f_{t,c}^{(n)},$ $g_{t,c}^{(n)}$ are Hermite equivalent but not $\GL_2(\Zz )$-equivalent.
\end{itemize}
\end{thm}
\begin{remark}
Note that, by Dirichlet's theorem on primes in arithmetic progressions, for given $n$ there are infinitely many choices for $c$ as in the statement of Theorem~\ref{mind},
and further that for given $n,c$ there are infinitely many choices for $t$.

If we fix $n,c$ and let $t\to\infty$, then the absolute value of the discriminant of $f_{t,c}^{(n)}$ tends to $\infty$, and thus, the pairs $f_{t,c}^{(n)}$, $g_{t,c}^{(n)}$ 
run through infinitely many different Hermite equivalence classes.
\end{remark}

\noindent
{\bf Acknowledgement.} We are very grateful to the anonymous referee, who very carefully scrutinized our paper and corrected some errors.

\bibliographystyle{abbrv}
\bibliography{bibfile-1}

\end{document}